%% file: v7-ARXIV.tex
\documentclass[aip,reprint]{revtex4-1}

\usepackage{amsmath, amsfonts, amsthm, amssymb, amscd, bbold, comment, enumerate, subcaption, dcolumn, bm, setspace}
\usepackage[utf8]{inputenc}
\usepackage[T1]{fontenc}
\usepackage{mathptmx}
\usepackage[all,cmtip]{xy}
\usepackage[dvips]{graphicx}
\usepackage{epsfig}
\usepackage{color}


\newtheorem{thm}{Theorem}
\newtheorem{cor}[thm]{Corollary}

\newtheorem{prop}[thm]{Proposition}
\newtheorem{conj}[thm]{Conjecture}

\begin{document}

%\preprint{AIP/123-QED}

\title{Typical knots: size, link component count, and writhe}
\author{Margaret Doig}
\affiliation{Department of Mathematics, Creighton University}
\email{margaretdoig@creighton.edu}
%supported by....

%\begin{singlespacing}

\begin{abstract}
We model the typical behavior of knots and links using grid diagrams. Links are ubiquitous in the sciences, and their \emph{normal} or \emph{typical} behavior is of significant importance in understanding situations such as the topological state of DNA or the statistical mechanics of ring polymers. We examine three invariants: the expected size of a random knot; the expected number of components of a random link; and the expected writhe of a random knot. We investigate the first two numerically and produce generating functions which codify the observed patterns: knot size is uniformly distributed and linearly dependent upon grid size, and the number of components follows a distribution whose mean and variance grow with $\log_2$ of grid size; in particular, for any fixed $k$, the $k$-component links grow vanishingly rare as grid size increases. Finally, we observe that the odd moments of writhe vanish, and we perform an exploratory data analysis to discover that variance grows with the square of grid size and kurtosis is constant at approximately 3.5. We continue this project in a future work,\cite{doigtypicalii} where we investigate genus and the effects of crossing change on it.
\end{abstract}

\maketitle

\section{Knots and links as outcomes in a probability space}

It is common (and enlightening) to study knots and links with extremal properties or to search for classifications and definite statements: \emph{Can I draw this particular knot with fewer crossings?} and: \emph{How many knots have a diagram with 7 crossings?} or: \emph{Is a minimal diagram for an alternating knot alternating?} and even: \emph{Can this invariant detect the unknot?} 

This is not our approach; rather, we pursue a complementary path, to understand the normal behavior of knots and links. We wish to investigate the family of links as a probability space and ask about the properties a typical or random link. There are many concrete topological reasons to study random links: \emph{If I doodle some 7-crossing knots, will there be interesting examples, or will I end up with a bunch of unknots?} or: \emph{If I study this invariant for pretzel knots, will it be representative of the behavior of knots in general?} and: \emph{Seifert genus basically grows with crossing number, but how fast, and how consistently?} or finally: \emph{If I implement my algorithm for this invariant, will it usually finish in a reasonable amount of time?}

Similar questions about the typical behavior of knots also arise from a number of applications: cellular processes are affected by the knotting and writhing of DNA (see Section~\ref{sect:dna_top}), the statistical mechanics of ring polymers seem to depend on their knot types (see Section~\ref{sect:polymers}), and even a new cryptographic scheme depends on whether the knot used as a key is sufficiently hard to approximate.\cite{knotcrypto}

\section{Motivations}

\subsection{Knots and links as models for DNA and other polymers}

Knots appear everywhere in nature: in polymers, DNA, networks, even headphone cords. We mathematically define a \emph{knot} to be a closed curve in 3-D space, or, more precisely, an embedding of the circle $S^1$ into the 3-sphere $S^3$; a \emph{link} is a set of component knots which may be entangled with one another in space. See Figure~\ref{fig:trefoil}. 

Knots and links model single-strand chain polymers, and they even model DNA reasonably well - even though DNA consists of two strands, they are complementary and twist into a single curve with rare effects for topology other than that they resist twisting (which we study in Section~\ref{sect:writhe}). Additionally, while eukaryotic DNA is technically not a closed loop as our definition of a knot requires, it is divided into large domains whose ends are fixed in the cellular structure away from the location of cellular processes and cannot be practically recruited to alter local topology. Therefore, we may reasonably model it with closed curves. Similarly, long linear polymers often locally look entangled in a way which meets our intuition of a knot, and we may justify modeling them with knots by imagining the ends to be fixed in space some distance away from the action.

We have selected the grid diagram as our basic model for knots and links: we place arcs of variable length within a grid subject to some minor technical constraints, and we add crossings to resolve singularities. Under Even-Zohar's classification of knot models,\cite{evenzoharmodels} grid diagrams fall into the category of models with variable-length segments. Surprisingly, they do not appear to be limited by their right-angled or superficially 2-D nature. There are good biological and physical reasons to consider this model to be sufficiently representative of many situations in which random knots appear in nature. Some of our questions below have previously been addressed for other models, mostly from the category of constant-length segments, which seem to display distinct behavior. See Section~\ref{sect:model} for details.

\begin{figure}
\resizebox{0.7\linewidth}{!}{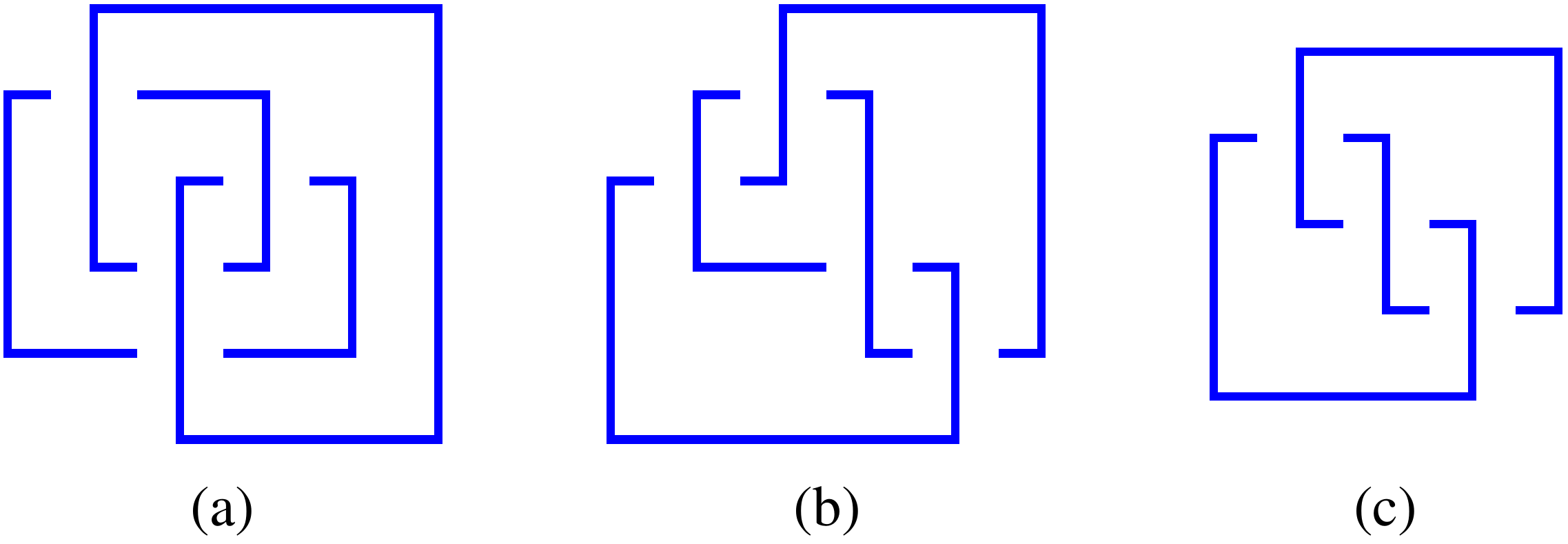}%preprint knots.pdf_t
\caption{\raggedright (a) An unusual diagram for the unknot; (b) the figure 8 knot; (c) the trefoil knot.}\label{fig:trefoil}
\end{figure}

\subsection{Polymer dynamics and knot invariants}\label{sect:polymers}

The topology of polymers appears to influence their statistical mechanics. For example, the complexity of knotting of a long polymer chain can affect its response to weak mechanical stretching,\cite{caraglio} its friction when passing through a constrained hole or pore,\cite{plesa2016direct} and its mobility when forced to pass through a gel or other resistant environment, as in electrophoresis.\cite{arsuaga2002knotting, cebrian2015electrophoretic, stasiak1996electrophoretic, olavarrieta2002supercoiling, trigueros2001novel, trigueros2007production, weber2006numerical, weber2006simulations} Knotting of a polymer ring can likewise affect the standard size of the ring, the equilibrium relaxation time, its general motion, and the diffusion constant, among other properties.\cite{orlandini2017statics} From the other direction, the formation of knotting and its complexity can be affected by degree of polymerization, quality of solvent, temperature, and confinement of the polymer.\cite{orlandiniwhittington}

Our first invariant, studied in Section~\ref{sect:knot_size}, is motivated by the idea of building a random knot one segment at a time until it closes when the ends collide. We reproduce this by constraining ourselves within an $n\times n$ grid and placing consecutive horizontal and vertical arcs of varying sizes (with some minor restrictions to maintain the integrity of the knot diagram) until the knot closes, and we count the number of arcs used. We find that it is uniformly distributed among all possible values (all even numbers from 4 to $2n$). A number of theoretical and experimental works have previously examined scaling behavior of knots such as the relative dimensions of the knot compared to the number of arcs which make it up,\cite{descloizeaux, orlandinitesi, deutsch, grosberg, shimamura, dobay, matsuda, lgm1, lgm2, uehara} and this effort grows out of the investigation of random walks, both self-avoiding and not.\cite{flory, degennes, doi} The previous theoretical results have primarily used models from the category of constant-length segments, which seem to behave differently from grid diagrams and other models using variable-length segments. 

Our second invariant, examined in Section~\ref{sect:numb_comp}, is motivated by the idea of encountering a random link of a given size and counting how many components make up the link. We model this by selecting an $n\times n$ grid diagram (see Section~\ref{sect:model}) and checking the number of components. As one may suspect, it is normally distributed, with mean and moments we derive from a generating function. Additionally, links of any given number of components become vanishingly rare as the link size grows. A similar question has previously been studied for random links in the braid and bridge presentation models,\cite{ma} which fall into the same category as grid diagrams, but this question has surprisingly has not been analyzed for most models, including grid diagrams. It is similar to but distinct from the question of counting cycle lengths within permutations.

\subsection{DNA topology and writhe}\label{sect:dna_top}
Our third and final invariant in Section~\ref{sect:writhe}, writhe, is motivated primarily by DNA topology. Writhe is the twisting or looping of a curve projected onto the plane. It reflects an effect most gardeners are familiar with: a carelessly coiled garden hose will resist being straightened out and will form kinks or supercoils. DNA molecules experience similar effects on a grander scale, and the resulting topology has a significant effect on cellular processes. 

DNA topology is typically described by three values, the \emph{linking number} (the total number of times the two strands link with each other), the \emph{twist} (the total number of complete helical turns of a pair), and the \emph{writhe} (which describes the greater topological structure, the number of times the helix crosses over itself, with sign). The linking number is invariant under local deformations of the DNA, while the twist and writhe are related to one another (under appropriate normalization, they sum to the linking number) and may be altered by cellular processes. The thermodynamically preferred relaxed B-DNA form has one full twist approximately every 10 base pairs, and torsional stress will tend to convert some of any over- or under-twisting into supercoiling of the molecule, or writhe. In nature, DNA is rarely found in its relaxed form and is on average underwound by about 6\%.\cite{vologodskiicozzarelli}

Genetic processes often alter the supercoiling or writhe of a molecule. Transcription involves the movement of RNA polymerase down the strand to unwind, transcribe, and rewind the twin strands. If unconstrained, the polymerase would follow the twist and proceed in a screw-like fashion down the helix. In the twin-domain supercoiling model of Liu and Wang,\cite{liuwang} though, the polymerase is not unconstrained, and transcription results in positive coiling ahead and negative coiling behind. While these supercoils do not alter the global topology from a mathematical point of view, they will not locally resolve one other in nature. Replication also results in coiling: when the DNA forks for replication, positive supercoils form ahead of the fork, and a precatenane, or positive winding of the two partially replicated strands, forms behind the fork. If not removed after replication, the precatenane will convert to a catenane, a linkage between the resulting sister chromosomes.\cite{champouxbeen}

This topological state of DNA significantly affects its availability for genetic processes. Negatively supercoiled molecules, as generally found in cells, favor reactions that require unwinding the helix such as transcription and replication, and positive supercoiling favors processes which require rewinding, such as the rescue of a stalled replication fork. The presence of a catenane itself is debilitating during cytokinesis. Topoisomerases generally police the topological state of the DNA, where they sever a strand, change a crossing, and reattach the strand, thereby removing a positive supercoiling after replication, eliminating a linkage between chromosomes, or otherwise maintaining the topological stability of the molecule.\cite{deweeseosheroff,espelimarians,seolneuman} Other enzymes appear to act by a related set of moves, where they replace one tangle by another.\cite{darcyvazquez} 

We model writhe by selecting a random knot in an $n\times n$ grid diagram and counting its crossings, with sign. We compare the observed writhe to both the size of the diagram $n$ and the actual length of the knot.

\section{Mathematical prerequisites}

The attentive mathematical reader will have noted that we frequently refer to a generic link as a knot, even though it may have more than one component. We beg the reader's indulgence for the sake of connecting to our audience's intuitive understanding of knotted objects in the world around us.

\subsection{Grid diagrams as a model for random links}\label{sect:model}
Numerous models are available for studying the typical or average knot. Following the naming convention of Even-Zohar,\cite{evenzoharmodels} the first set, which can be called the set of 1-D models, includes random walks in grids and random polygonal walks, both of which involve successively adding small steps of standardized size and which are inclined towards local knotting and easily generate knots and links which are connected sums. The next set, the 2-D models, includes models which are in some sense less linear, like random planar diagrams or planar curves, or the knot tables themselves; these are still inclined towards non-prime knots (except for the knot tables, where these have been deliberately removed) and satellite knots (ex, double figure 8 configurations are common). The last set of the so-called 3-D models includes random jumps (whose difference from polygonal walks is that the length of each successive edge may vary), the Petaluma model, and the grid diagram. In each of these, the successive steps tend to be comparable to the knot diameter, and it is conjectured that they are overwhelming prime and hyperbolic and that various invariants obey distributions which seem intuitively reasonable.\cite[Section 5.4]{evenzoharmodels}

It is not clear which category DNA knots and other polymers fall into. If one imagines a polymer forming by the slow addition of successive base pairs or other small units to a longer chain, then an ideal model would be from the 1-D category, such as Brownian motion or lattice walks. On the other hand, 3-D category models show some promise for modeling knot formation by topoisomerase and other cellular actors which spontaneously change crossings or tangles or otherwise alter a preexisting knot or link.\cite{barbensi2020grid} 

We have selected the \emph{grid diagram} as our model. A grid diagram of size $n$ as in Figure~\ref{fig:grid_diagram} consists of an $n\times n$ grid drawn on a torus with $n$ black dots and $n$ white dots filled in so that each row and each column has a single black dot and a single white dot, in different cells. A grid diagram corresponds to a link: Insert an arc in each column from the black dot to the white dot; in each row, insert an arc from the white dot to the black dot. When two arcs intersect, replace the intersection by a crossing with the vertical arc on top. 

There are three \emph{Cromwell moves} corresponding to the Reidemeister moves which change the appearance of the diagram but do not change the type of the underlying knot or link: commutations (where two arcs in successive rows or columns are exchanged); translations (where the arc in the top row is moved to the arc in the bottom, or the left to the right); and stabilizations/destabilizations (there are four types). Figure~\ref{fig:grid_diagram} shows a trefoil and its corresponding grid diagram. See Cromwell\cite{cromwell} for a thorough introduction to grid diagrams. 

\begin{figure}
\begin{center}
\resizebox{0.75\linewidth}{!}
{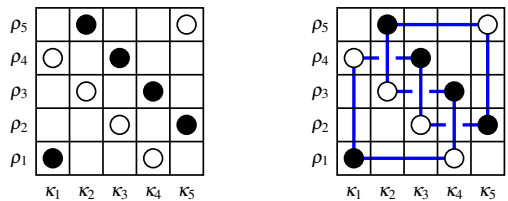}
\end{center}
\caption{\raggedright  A grid diagram and corresponding knot diagram for the trefoil knot. The edges may be oriented so vertical arcs are directed from the black dot to the white dot; in other words, the left most arc points upwards.}\label{fig:grid_diagram}
\end{figure}

%preprint
\begin{comment}
\begin{figure}
\begin{center}
\begin{subfigure}[b]{0.5\textwidth}
\resizebox{0.25\linewidth}{!}{\input{grid_diagram1.pdf_t}}
\caption{\raggedright }\label{fig:grid_diagram1}
\end{subfigure}
\begin{subfigure}[b]{0.5\textwidth}
\resizebox{0.25\linewidth}{!}{\input{grid_diagram2.pdf_t}}
\caption{\raggedright }\label{fig:grid_diagram2}
\end{subfigure}
\end{center}
\caption{\raggedright  A grid diagram and corresponding knot diagram for the trefoil knot from Figure~\ref{fig:trefoil}(c). The edges may be oriented so vertical arcs are directed from the black dot to the white dot; in other words, the left most arc points upwards.}\label{fig:grid_diagram}
\end{figure}
\end{comment}

\subsection{Combinatorial details}\label{sect:comb}
We have several choices to encode a grid diagram combinatorially. For example, we may record the row and column coordinates of each dot; for simplicity, we could record the  column coordinates only, reading from the top row to the bottom, which would give two permutation of the numbers $1, 2, \cdots, n$ (equivalently, two $n$-cycles from $S_n$). This serves well for describing any known grid diagram, but it is challenging to generate a random diagram this way: two such permutations would generate a set of black and white dots, a pair for each row and a pair for each column, but we could have collisions between the black and white dots. In Sections~\ref{sect:numb_comp} and \ref{sect:writhe}, we will use this method to generate random links: we will select two permutations of $1, 2, \cdots, n$ and eliminate any pair that induces collisions. Equivalently, we could select a permutation to give the column coordinates of the black dots reading from top to bottom, and then a derangement (a permutation which does not leave any element in its original position) which tells us the column for a white dot based on the column for each black dot. Using the latter description, since there are $n!$ permutations and $!n$ derangements of $n$ elements, we see there are $n!\cdot !n$ possible $n \times n$ grid diagrams.

In the case of knots, we could select a starting point at some black dot on the knot and then trace through the knot, recording in parallel the order in which the knot visits the columns and the rows; in this case, we again record two permutations. Observe that any knot corresponds to exactly $n$ different permutations, depending on which black dot was selected as the starting point. We may easily randomly generate knots by selecting two permutations. We use this encoding or a slight variation for the knots in Sections~\ref{sect:knot_size} and \ref{sect:writhe}. Using this description, we may easily see that there are $n!(n-1)!$ distinct knot diagrams in an $n\times n$ grid.

\subsection{Implementation}

All calculations were performed within the author's toolkit and are available to the public for further experimentation.\cite{dmt} Random permutations were generated using the default\_random\_engine in c++14, seeded with system clock time. Data analysis (including all curve fitting) was performed in JMP, except for the calculations of higher moments and kurtosis, which were performed in Microsoft Excel; knot and grid diagrams were generated in xfig, sample mass functions in Excel, and other graphs in JMP.

\section{The size of a typical knot}\label{sect:knot_size}

We warm up by studying the size of a knot randomly drawn within an $n\times n$ grid. As a polymer might form and then close up in 3-D space, we envisage a knot building up arc by arc inside a grid diagram, but we allow the knot to close off at any time that the end touches the beginning. In particular, we do not require that the knot fill out the entire grid diagram. We will see that the knot size is uniformly distributed, which supports the use of grid diagrams to model polymeric behavior: if a polymer's building blocks are of width $1$ and variable length, and we wish to study a polymer constrained within an $n\times n$ area, then we should model the polymer by looking at an $n\times n$ grid diagram. As expected, the polymer is equally likely to close up after any given number of building blocks are added, within the constraint imposed by the area.

We model the question numerically, and then we produce a probability generating function to summarize behavior and verify the reasonableness of our model. One may argue that the combinatorics of this example are obvious; nevertheless, we produce the simulation to argue for the applicability of the grid diagram model to questions of polymer topology. 

\subsection{Numerical simulation}
Recall the combinatorial encoding of a knot from Section~\ref{sect:comb}: an $n\times n$ knot is represented by two $n$-permutations $\rho = (\rho_1~\rho_2 ~\cdots ~ \rho_n)$ and $\kappa=(\kappa_1~\kappa_2~\cdots~\kappa_n)$ which tell us the order in which we visit the rows and the columns, respectively; that is, black dots would be located at the points with coordinates $(\rho_i, \kappa_i)$ as the white dots would be located at the $(\rho_{i+1}, \kappa_i)$. To generate random knots living inside an $n\times n$ grid, we assume without loss of generality that $\rho_1 = 1$ and then select two random $n$-permutations $\rho$ and $\kappa$. We read off $\kappa_1$ from the front of $\kappa$, then $\rho_2$ from $\rho$, and so on; note that the $\kappa_i$s will not repeat, and the $\rho_i$s will not repeat with the possible exception that some $\rho_s$ may be 1, in which case we will declare that the knot closes up at size $s$. (We omit the artificial case where $\rho_2 = 1$ and the knot is empty.) This is not exactly the combinatorial encoding for a grid diagram presented in Section~\ref{sect:comb}, but it could be converted easily by removing the remaining $n-s$ unused row/column numbers.

\begin{figure}
\begin{subfigure}[b]{0.45\textwidth}
\includegraphics[width=\textwidth]{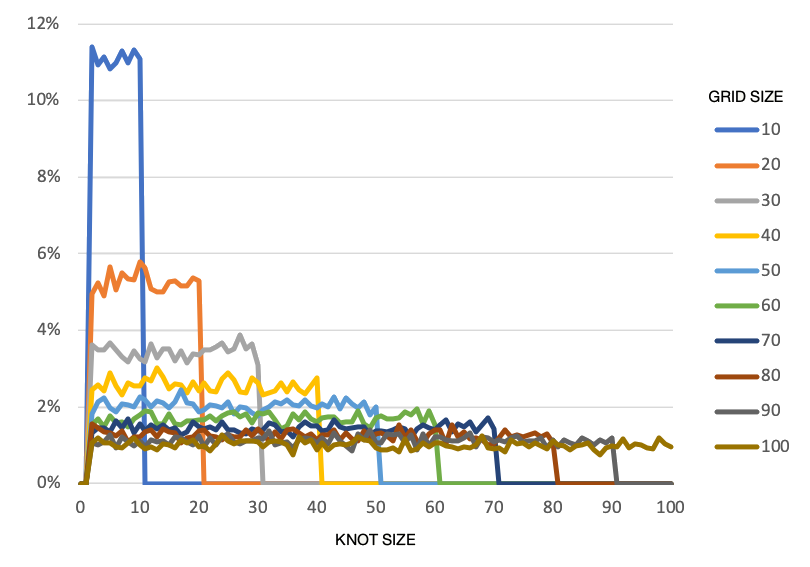}
\caption{\raggedright } \label{fig:knot_size_1}
\end{subfigure}
\begin{subfigure}[b]{0.45\textwidth}
\includegraphics[width=\textwidth]{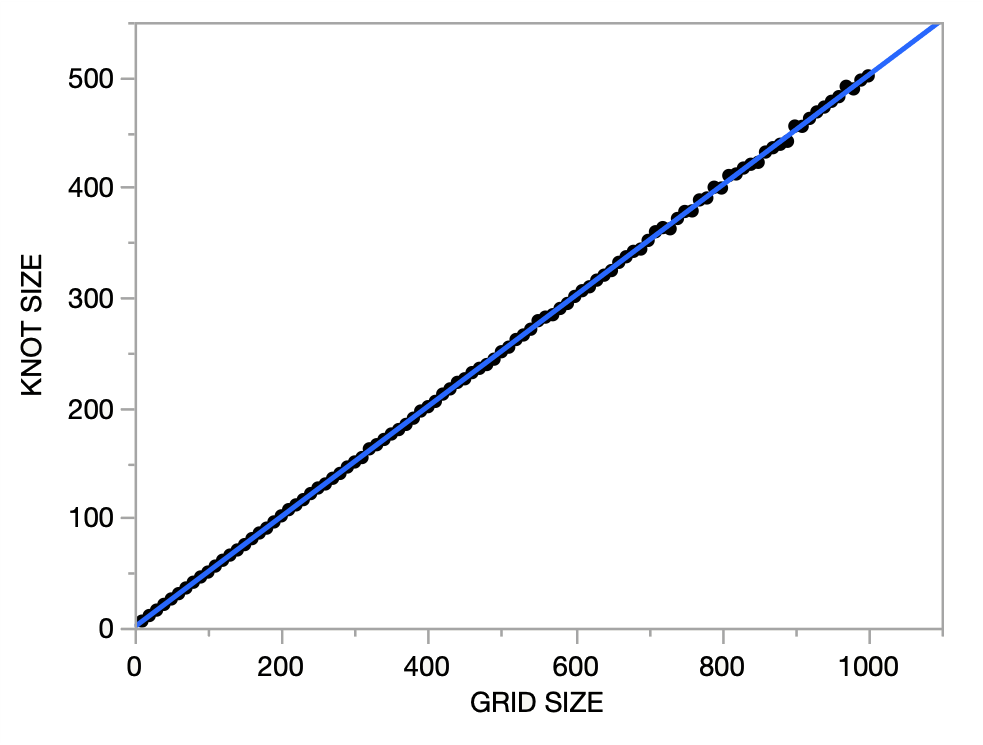}
\caption{\raggedright } \label{fig:knot_size_2}
\end{subfigure}
\caption{\raggedright The size of knots randomly drawn in an $n \times n$ grid: (\subref{fig:knot_size_1}) the sample mass functions for $n \leq 100$; (\subref{fig:knot_size_2}) mean knot size compared to grid size for $n \leq 1000$.}\label{fig:knot_size}
\end{figure}

Randomly generating 10,000 knots in an $n \times n$ grid for each of $n=10, 20, \cdots, 1000$, we see in Figure~\ref{fig:knot_size_1} that the size $s$ of these knots is uniformly distributed from $2$ to $n$. We may fit the average knot size to the grid size $n$ as in Figure~\ref{fig:knot_size_2}:
\[\overline{s}(n) = 0.7291814 + 0.50085096n\] 
with $R^2 = 0.999879$. 

\subsection{Mathematical analysis}

Conveniently, we may directly calculate the expected size of a knot randomly drawn in an $n\times n$ grid, and it agrees with our experimental estimate and expands our observations:

\begin{thm}\label{thm:knot_size}
Consider a knot drawn randomly inside an $n \times n$ grid. The knot size is uniformly distributed from 2 to $n$. That is, the probability $p_s$ that the knot will be size $s$ is:
\[p_s = \frac{1}{n-1}, \qquad 1 < s \leq n.\]
Then the expected value of the knot size is
\[EV(s) = \frac{n}{2}+1,\]
the variance is 
\[Var(s) = \frac{n^2-2n}{12},\]
and the $k^{th}$ moment is
\[\mu'_k(s) = \frac{1}{n-1}\sum_{i=2}^n i^k.\]
\end{thm} 

\begin{proof}
We see immediately that the probability of a $1$ appearing in the $s^{th}$ position in a random permutation of $1, 2, \cdots, n$ is exactly $\frac{1}{n}$; omitting the cases where $1$ appears in the first term, we have $p_s = \frac{1}{n-1}$.

Observe that $\sum p_s = 1$ as it should be, and the expected value of the knot size is therefore
\[EV(s) = \sum_{s=1}^n sp_s = \frac{1}{n-1}\sum_{s=2}^n s = \frac{n}{2}+1.\]
We observe also that we could establish an (ordinary) probability generating function for $p_s$; that is, we can build a function $P(t) = \sum_{k=0}^\infty p_s t^s$ which neatly encapsulates all the $p_s$ within a single series. We can recover the $p_s$ either from the series expansion of $p_s$ or by noting $\frac{\partial^sP}{\partial x^s}(0) = s!p_s$. In this case:
\[P(t) = \frac{1}{n-1} \sum_{s=2}^n t^s = \frac{t^2 + t^3 + \cdots + t^n}{n-1} = \frac{t^2(1-t^{n-1})}{(n-1)(1-t)}.\]
Recall that an exponential generating function $g(x)$ for a sequence $a_k$ is a series expansion $g(x) = \sum_{k=0}^\infty a_k \frac{x^k}{k!}$; it is an elementary fact that, if $P(t)$ is a probability generating function for a variable, then $P(e^t)$ is an exponential generating function for its (non-central) moments $\mu'_k$, 
\begin{multline*}
P(e^t) = \frac{e^{2t}+e^{3t}+\cdots +e^{nt}}{n-1} \\ %preprint
= \frac{1}{n-1} \left(\sum_{k=0}^\infty \frac{2^kt^k}{k!} + \sum_{r=0}^\infty\frac{3^kt^k}{k!} + \cdots + \sum_{k=0}^\infty\frac{n^kt^k}{k!}\right)\\
= \frac{1}{n-1}\sum_{k=0}^\infty \left(2^k+3^k+\cdots n^k\right)\frac{t^k}{k!}
\end{multline*}
so
\[\mu'_k(s) = \frac{2^k+3^k+\cdots n^k}{n-1}.\]
This formula agrees with our direct calculation of $\mu'_1(s)$, a.k.a. $EV(s)$, and likewise gives a formula for variance (recall $Var(s) = EV(s-EV(s))^2 = \mu_2'-(\mu_1')^2$) in agreement with the standard formula for a discrete uniform distribution.
\end{proof}

\section{The number of components of a typical link}\label{sect:numb_comp}

We next consider random links. These model the behavior of self-assembling polymers, which are not guaranteed to form a single chain when left to their own devices. We first consider the number of components $\#$ of such a link in an $n\times n$ grid diagram.

We use the second combinatorial encoding of Section~\ref{sect:comb} of generating two $n$-permutations to give the $y$-coordinates for the black and white dots and removing any impossible grid diagrams (i.e., pairs of permutations with collisions). We count the number of valid link diagrams $c_n$ and the number with $k$ components $c_{n,k}$. There are na\"ively $n!n!$ ways to place the black and white dots in a grid without accounting for collisions, and so we will also make use of a renormalization, $\overline c_n = \frac{c_n}{n!n!}$ and $\overline c_{n,k} = \frac{c_{n,k}}{n!n!}$.

Note that this is related to but not the same as asking for the number of cycles of length $k$ in an $n$-permutation; following the first model from Section~\ref{sect:comb}, we could list a permutation for the black dots and count the number of cycles of any given length, but this does not give a one-to-one identification between cycles and grid diagrams: a grid diagram also requires a choice of permutation with matching cycle sizes for the white dots; permuting the order of the disjoint cycles may alter the diagram but will not alter the permutation; and we forbid cycles of length 1 (though we resolve this conflict by using a derangement rather than a permutation). This is an instructive comparison, however, because initial behavior of these invariants does resemble that of random permutations, and so we do contrast our results with those for permutations.

\subsection{Numerical simulation}

\begin{table}
\[\small\begin{array}{c | r r r | r | l}
\hline\hline
n & c_{n,1} & c_{n,2} & c_{n,3} & c_n & \multicolumn{1}{c}{\overline{k}} \\
\hline
2 & 2 & & & 2 &1\\
3 & 12 & & & 12 & 1 \\
4 & 144 & 72 & & 216 & 1.3333\\
5 & 2,880 & 2,400 & & 5,280 & 1.4545\\
6 & 86,400 & 93,600 & 10,800 & 190,800 & 1.6038\\
7 & 3,628,800 & 4,656,960 & 1,058,400 & 9,344,160 & 1.7249\\
\hline\hline
\end{array}\]
\caption{}\label{table:count_comps}
\end{table}

To explore the number of components $k$, we generate all possible grid diagrams up to $n=7$, where computation becomes prohibitive. We count the number of knots and 2- and 3-component links ($c_{n,1}$, $c_{n,2}$, and $c_{n,3}$, respectively), and we calculate the average number of components $\overline{k}$ in Table~\ref{table:count_comps}.

\begin{figure}
\begin{subfigure}[b]{0.45\textwidth}
\includegraphics[width=\textwidth]{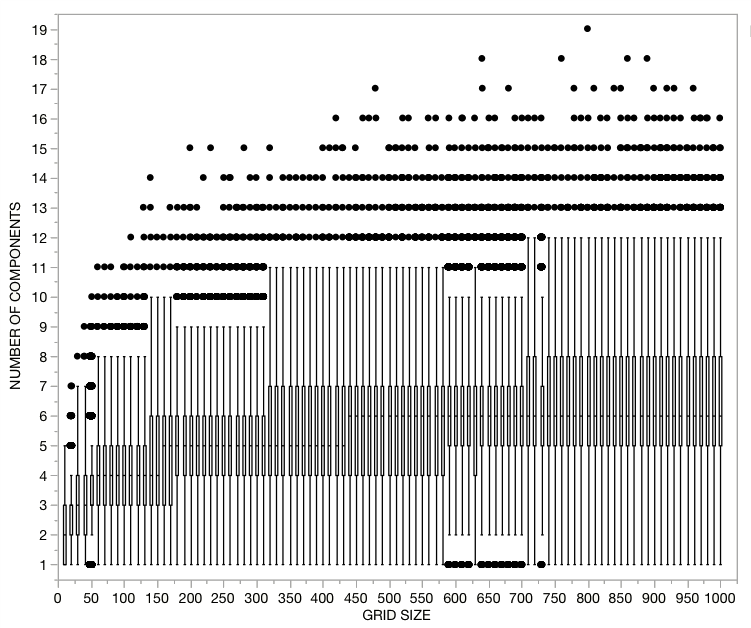}
\caption{\raggedright }\label{fig:count_comps}
\end{subfigure}
\begin{subfigure}[b]{0.45\textwidth}
\includegraphics[width=\textwidth]{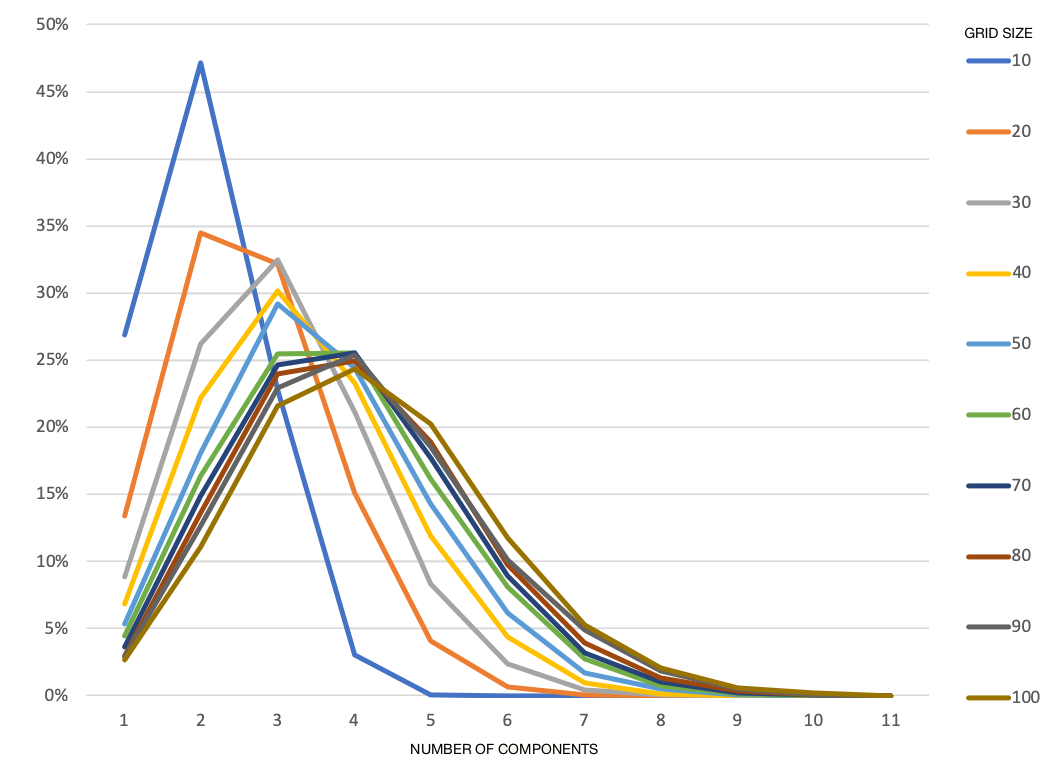}
\caption{\raggedright }\label{fig:count_comps_2}
\end{subfigure}
\caption{\raggedright The number of components in a randomly drawn $n \times n$ link grid diagram: (\subref{fig:count_comps}) 
box plots for $n\leq 1000$; (\subref{fig:count_comps_2}) sample mass functions for $n \leq 100$.}\label{fig:count_comps}
\end{figure}

To explore component size for larger $n$, we generate 10,000 grid diagrams for each size $n=10, 20, \cdots, 1000$ and count the number of components of each. Each grid diagram comes from two random permutations (generated by the default\_random\_engine in c++14 seeded with clock time) of the numbers $1$ through $n$, after eliminating ones which generated collisions between the black and white dots. The number of components is distributed as seen in Figure~\ref{fig:count_comps}, with sample mean and variance:
\[\begin{split}
\overline{k} &= -0.330551 + 0.6831938\log_2 n\\
\sigma^2 &= -1.806079 + 0.6650469\log_2 n
\end{split}\]
with $p < 0.0001$ for both. We have selected the binary logarithm in part informed by Theorem~\ref{thm:cnk} and other results below highlighting similarities between the behavior of the $c_{n,k}$ and the Harmonic numbers, which obey a classic bound in terms of $\log_2(n)$. Selecting the natural log results in very similar models. 

\begin{figure}
\begin{subfigure}[b]{0.45\textwidth}
\includegraphics[width=\textwidth]{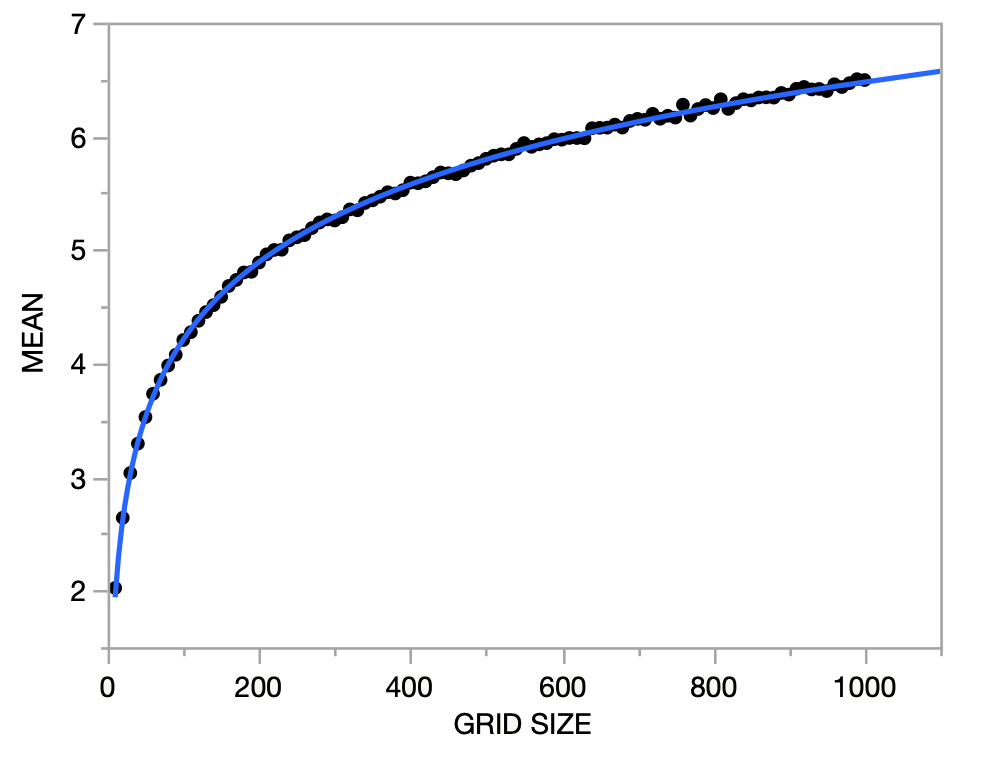}
\caption{\raggedright }\label{fig:mean_by_grid_size}
\end{subfigure}
\begin{subfigure}[b]{0.45\textwidth}
\includegraphics[width=\textwidth]{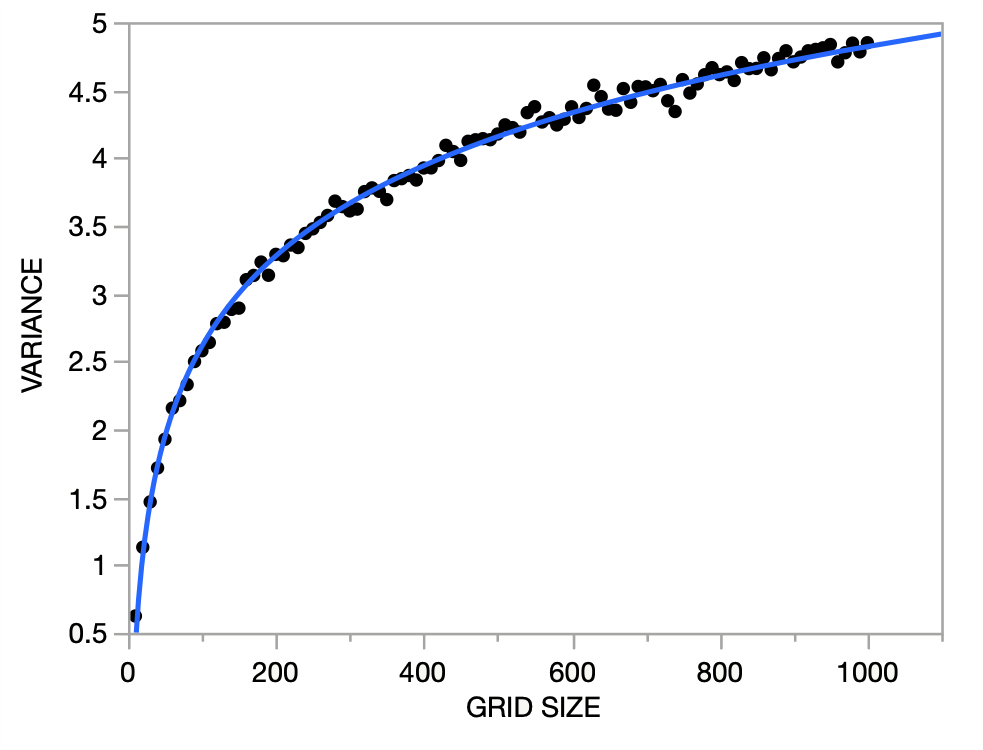}
\caption{\raggedright }\label{fig:var_by_grid_size}
\end{subfigure}
\caption{\raggedright The fit of (\subref{fig:mean_by_grid_size}) mean and (\subref{fig:var_by_grid_size}) variance to grid size.}\label{fig:count_comps_mean_sd}
\end{figure}

Additionally, we observe that, as the grid size increases, knots become vanishingly rare; in fact, for any fixed $k$, the $k$-component links become vanishingly rare. If $\Pr(k \mid n)$ is the probability that an $n \times n$ grid diagram represents a $k$-component link, then
%\[\displaystyle\lim_{n\rightarrow \infty} \Pr(k \mid n) = 0.\]
\[\displaystyle \Pr(k \mid n) \rightarrow 0 \text{ as } n\rightarrow \infty\]
%\[\displaystyle\lim_{n\rightarrow \infty} \frac{c_{n,k}}{c_n} = 0.\]

\subsection{Mathematical analysis}

We may also calculate directly how many $n \times n$ grid diagrams have a given number of components. We have moved a number of proofs to Section~\ref{sect:gf} for readability.

We begin by restating a result likely known in the community and elegantly presented by Witte\cite[Theorems~5.1.1-5.1.2]{witte2019link}, the number of knot grid diagrams and of total grid diagrams of any given size. When we introduced knot and link notation in Section~\ref{sect:comb}, we mentioned these results and a justification similar to Witte's, and we here state the results again for later use and show how they follow from our generating functions of Theorem~\ref{thm:gf}, which is both satisfying and demonstrates the manner in which one may use these generating functions.

\begin{prop}\label{thm:cn}
For any $n$, the total number of $n\times n$ grid diagrams for knots is:
\[c_{n,1} = n!(n-1)!\]
The total number of link diagrams is:
\[c_n = n!\cdot !n = n!n! \sum_{i=0}^n(-1)^i\frac{1}{i!}.\]
\end{prop}

Equivalently, we could say 
\[\overline c_{n,1} = \frac{1}{n} \qquad \text{and} \qquad \overline c_n = \sum_{i=0}^n(-1)^i\frac{1}{i!}.\]
Note that $\{c_n\}$ is sequence A082491 in the On-Line Encyclopedia of Integer Sequences, and $\{c_{n,1}\}$ is equivalent to sequence A010790.\cite{oeis}

\begin{proof}
Recall that a generating function for $\overline c_n$ is some function $g(x) = \sum_{n=0}^\infty \overline c_n x^n$. Then we may recover its value by setting $\overline c_n = \frac{1}{n!}\frac{\partial ^n g}{\partial x^n}(0)$. Therefore, using the generating functions derived in Theorem~\ref{thm:gf},
\begin{multline*}
\overline c_n =\frac{1}{n!}\frac{\partial ^n}{\partial x^n}\left[\frac{1}{1-x}e^{-x}\right]_{x=0} \\ %preprint 
= \frac{1}{n!}\sum_{i=0}^n{n \choose i} \frac{\partial^{n-i}}{\partial x^{n-i}}\left[\frac{1}{1-x}\right]_{x=0}\frac{\partial^i}{\partial x^i}\left[e^{-x}\right]_{x=0}\\
= \frac{1}{n!} \sum_{i=0}^n\frac{n!}{i!(n-i)!}\left[\frac{(n-i)!}{(1-x)^{n-i+1}}\right]_{x=0}\left[(-1)^ie^{-x} \right]_{x=0} \\ %preprint
= \sum_{i=0}^n(-1)^{i}\frac{1}{i!}
\end{multline*}
and, if $n>1$, 
\begin{multline*}
\overline c_{n,1} =\frac{1}{n!}\frac{\partial ^n}{\partial x^n}\left[-\ln(1-x)-x\right]_{x=0} \\ %preprint
= \frac{1}{n!} \left[\frac{(n-1)!}{(1-x)^n}\right]_{x=0}
= \frac{1}{n}
\end{multline*}
\end{proof}
%WARN - typo in EV (seems to actually be H_{n-i}+1??)

We immediately see:

\begin{cor}\label{thm:e}
The total number of $n\times n$ grid diagrams $c_n$ displays asymptotic behavior similar to $n!n!$:
\[\lim_{n\rightarrow \infty} \frac{c_n}{n!n!} = \frac{1}{e}.\]
\end{cor}\label{thm:e}

It is more difficult to calculate explicitly the number of $k$-component links for some given $k > 1$, as may be seen from Witte's summation formula.\cite[Theorem~5.1.3]{witte2019link}. That said, we may bound the number and verify that they are ever sparser as the diagram enlarges.

\begin{thm}\label{thm:cnk}
For any $n$ and $k$, the number of $n\times n$ grid diagrams representing $k$-component links obeys:
\[c_{n,k} \leq n!(n-1)!\left(\log_2 n\right)^{k-1}.\]
\end{thm}

This generates a loose variation of the traditional Frisch-Wasserman-Delbr\"uck Conjecture \cite{delbruck1962knotting, frisch1961chemical}, which in its original form said that the unknot grows vanishingly rare among all knots as their size grows. The conjecture was first proven for polygonal knots\cite{Sumners_1988,pippenger1989knots} and later for a number of other knot models, most relevantly by Witte\cite{witte2019link}, who showed it in the grid diagram model for any given knot type.
\begin{cor}\label{cor:fwdconj}
For any fixed $k$, $k$-component links become vanishingly rare as the grid size $n$ grows,
\[\Pr(k \mid n) \rightarrow 0 \text{ as } n\rightarrow \infty.\]
%\[\lim_{n\rightarrow \infty} \frac{c_{n,k}}{c_n} = 0.\]
Indeed, any given link type $[L]$ is vanishingly rare among all links as $n$ grows. 
\[\Pr \left( [L] \mid n \right) \rightarrow 0 \text{ as } n\rightarrow \infty.\]
%\[\lim_{n\rightarrow \infty} \frac{[L]_n}{c_n} = 0.\]
\end{cor}

In general, we may produce a generating function for both $\overline c_n$ and $\overline c_{n,k}$, which is essential in establishing all the other theorems here.

\begin{thm}\label{thm:gf}
$\overline c_n$ has generating function
\[g(x) =(1-x)^{-1}e^{-x}\]
and $\overline c_{n,k}$ has generating function
\[G(x,y) =(1-x)^{-y}e^{-xy}.\]
If we fix $k$ and vary $n$, then $G(x,y)$ simplifies to
\[h(x)=(-1)^k\frac{(\ln(1-x)+x)^k}{k!}\]
On the other hand, if we fix $n$ and vary $k$, it simplifies to
\[f(y) =\sum_{i=0}^n(-1)^i{y+n-i-1 \choose y-1}\frac{y^i}{i!}.\]
\end{thm}

At last, we verify our observation of mean and variance.
\begin{thm}\label{thm:last_thm}
In an $n\times n$ grid, the expected value of the number of components $k$ is
\[EV(k) = \frac{1}{\overline c_n}\sum_{i=0}^{n-1}(-1)^i\frac{H_{n-i}-1}{i!}\]
and
\[Var(k) =\frac{1}{\overline c_n}\sum_{i=0}^{n}(-1)^i\frac{H^2_{n-i}-H_{n-i,2} + (2i-1)H_{n-i} + i(i-2)}{i!}\]
\end{thm}

Here the $H_{n,m} = \sum_{i=1}^n \frac{1}{i^m}$ are the generalized Harmonic numbers and $H_n = H_{n,1}$ the pure Harmonic numbers.

\subsection{Other link models}

Theorem~\ref{thm:cnk} is reminiscent of Ma\cite{ma}, who studied the random braid model, where a link $w_{n,k}$ is formed by closing up a random walk of $k$ steps on the braid group $\mathcal{B}^n$. Ma found that, if $n$ is fixed and $k$ allowed to pass to infinity, then the expected value of the number of components approaches $H_n$. Just as in our model, $EV$ also grows with $n$, although he does not derive separate results for the prevalence of a given number of components. He produces similar results for the bridge model, where a link $w_{2n,k}$ comes from closing up a random walk of $k$ steps in the mapping class group $\mathcal{M}_{0,2n}$ via an $n$-bridge presentation.

\section{Writhe $w$}\label{sect:writhe}

We finally approach \emph{writhe}, which is algebraic count of the crossings (i.e., the number of positive crossings minus the number of negative crossings - see below). 
\[
\xymatrix{
&& \\
\ar[ur] & \ar[ul]  |\hole_>>>>{+}
}
\qquad
\xymatrix{
&& \\
\ar[ur] |\hole & \ar[ul]_>>>>{-}
}
\]
The writhe $w$ of a knot projected into the plane is dependent upon its projection, or its diagram; however, if the knot is framed or is presented as a pair of adjacent strands (as in the case of DNA), then a change to writhe is reflected in a change of framing or of the twist of the pair of strands.

\subsection{Numerical simulation} 

We check the writhe of a random $n\times n$ knot diagram. Observe that Witte\cite{witte2019link} also performed a numerical simulation of writhe in grid diagrams, but he explored the distribution of writhe for a random diagram representing a fixed knot type, with obviously very different results.

The distribution of writhe for a random knot should be symmetric around 0 (every diagram has an equally likely mirror diagram where all crossings are changed; for grid diagrams, reflecting right-to-left is equivalent to changing all crossings), and the expected value for all the grid diagrams of any fixed size $n\times n$ is:
\[EV(w) = 0.\]
Similarly, skewness and all other odd higher moments will be 0, and there will be no difference between central and non-central moments.

\begin{figure}
\begin{subfigure}[b]{0.45\textwidth}
\includegraphics[width=\textwidth]{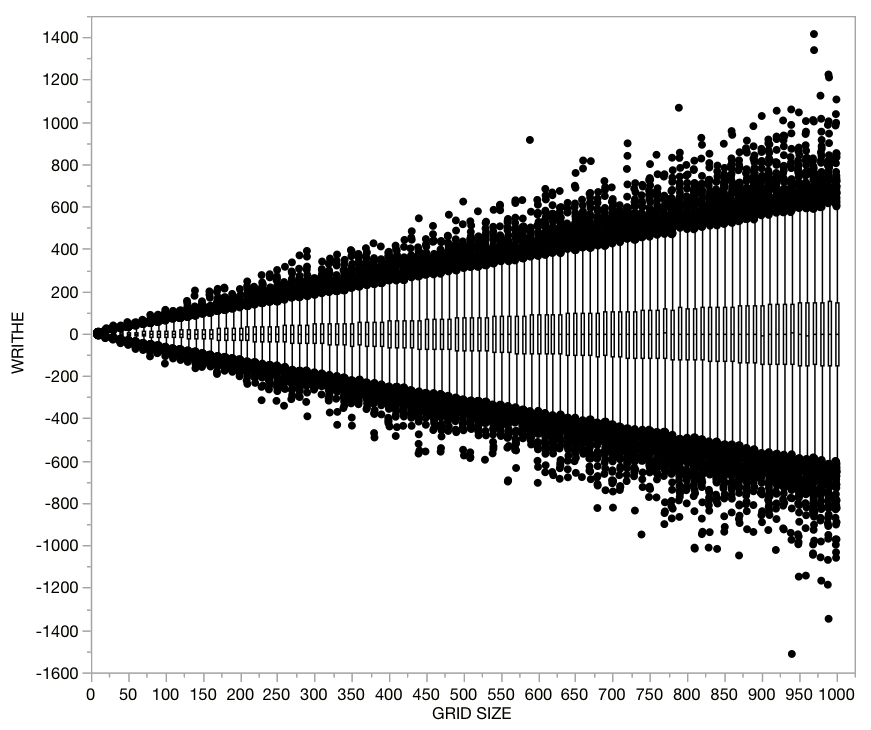}\caption{\raggedright }\label{fig:writhe_by_grid_size_1}
\end{subfigure}
\begin{subfigure}[b]{0.45\textwidth}
\includegraphics[width=\textwidth]{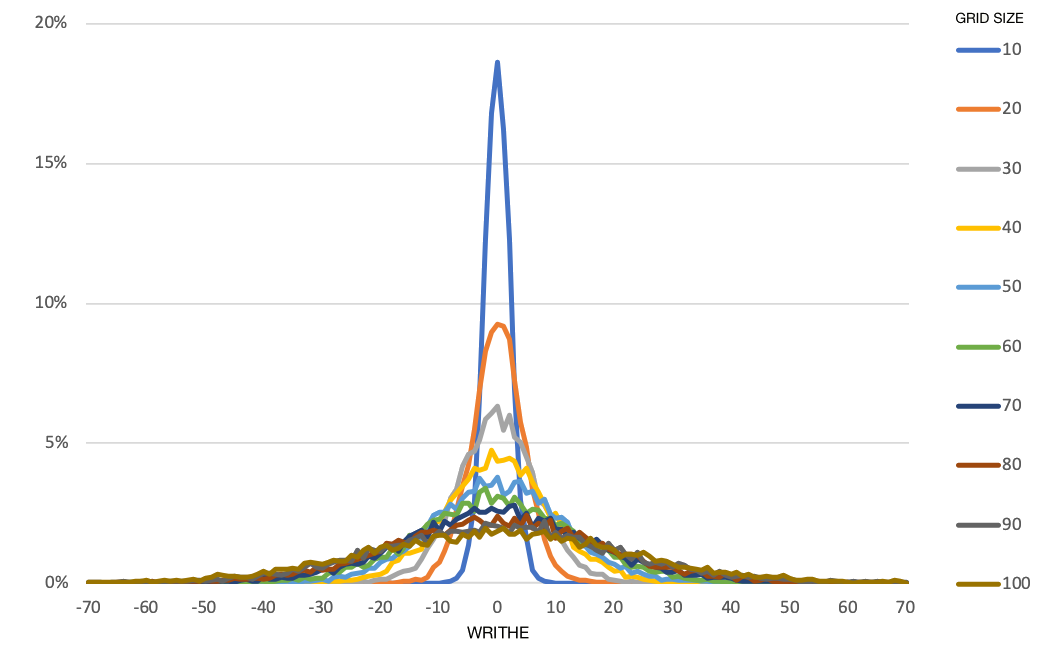}\caption{\raggedright }\label{fig:writhe_by_grid_size_2}
\end{subfigure}
\caption{\raggedright Writhe vs grid size for randomly generated $n \times n$ knot grid diagrams: (\subref{fig:writhe_by_grid_size_1}) box plots for $n\leq 1000$; (\subref{fig:writhe_by_grid_size_2}) sample mass functions for $n\leq 100$.}\label{fig:writhe_by_grid_size}
\end{figure}

\begin{figure}
\begin{subfigure}[b]{0.45\textwidth}
\includegraphics[width=\textwidth]{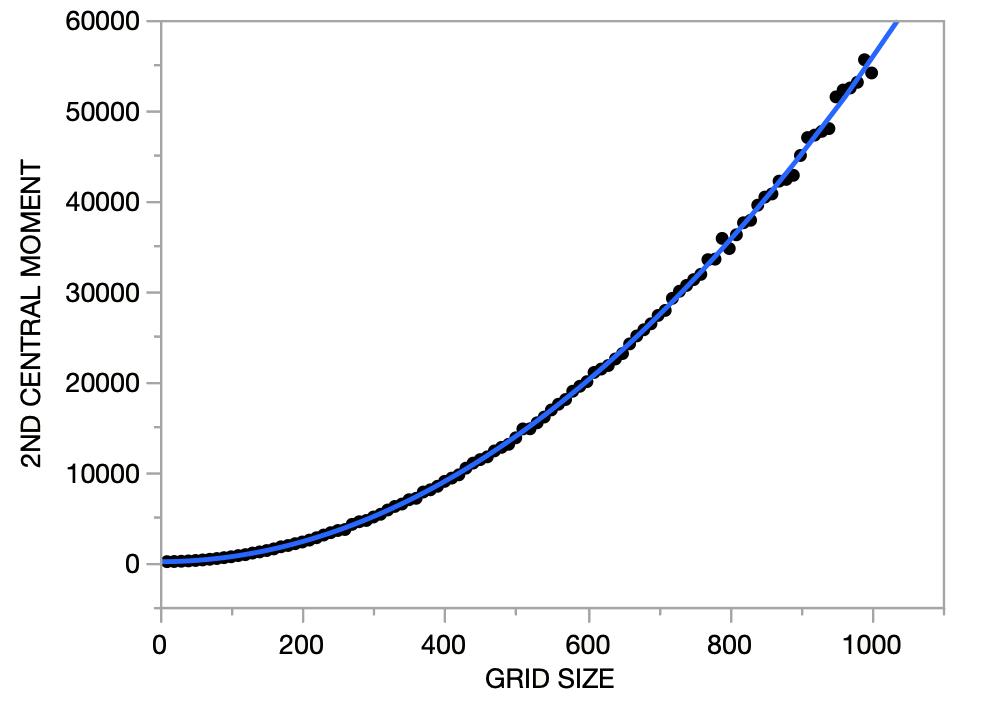}\caption{\raggedright}\label{fig:mu2_by_grid_size}
\end{subfigure}
\begin{subfigure}[b]{0.45\textwidth}
\includegraphics[width=\textwidth]{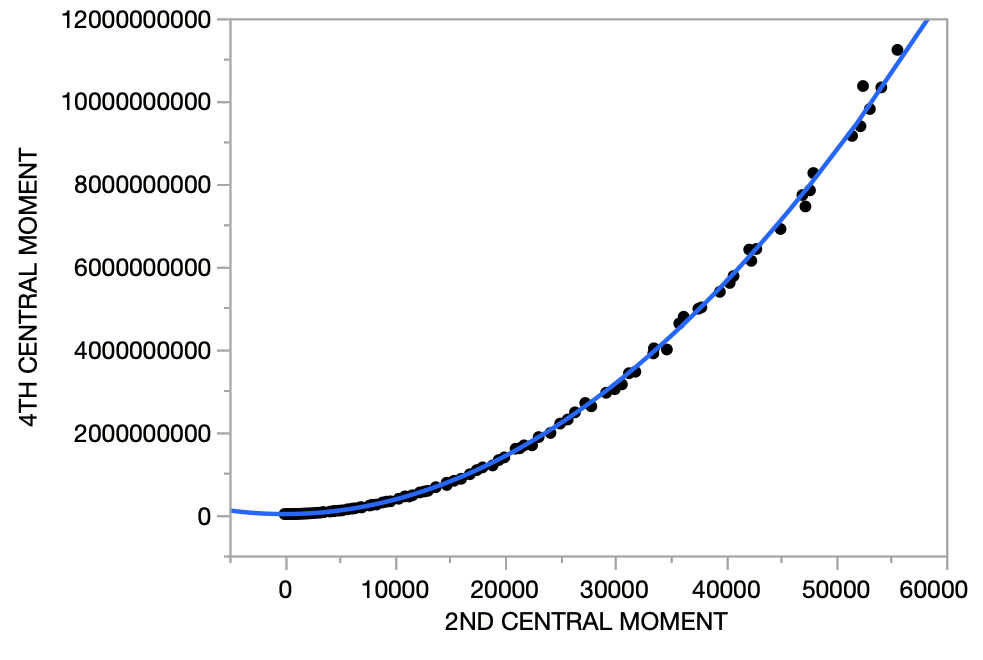}\caption{\raggedright}\label{fig:mu4_by_mu2_by_grid_size}
\end{subfigure}
\caption{\raggedright Writhe vs grid size for randomly generated $n \times n$ knot grid diagrams: (\subref{fig:mu2_by_grid_size}) variance vs. grid size and (\subref{fig:mu4_by_mu2_by_grid_size}) fourth moment $m_4$ vs. variance for $n\leq 1000$.}\label{fig:writhe_by_grid_size2}
\end{figure}

We consider the second description from Section~\ref{sect:comb}: a knot is two $n$-permutations which read off the order in which the columns and rows and encountered. First, we ran 10,000 cases for each of $n = 10, 20, \cdots, 1000$ and compared writhe to grid size; as shown in Figure~\ref{fig:writhe_by_grid_size}, the sample means are approximately 0 and the sample standard deviations are proportional to $n$. 

Our sample means ranged from -5.0709 to 3.1659 with an overall average of 
\[\overline{w} = -0.187818\]
and a tight grouping near 0; in fact, 93 of out 100 sample cases had a 95\% confidence interval containing $0$. We next wish to understand higher moments. Figure~\ref{fig:writhe_by_grid_size2} suggests that sample variances are proportional to $n^2$, so we fit $\sigma^2$ to $n^2$ using a linear model and forcing the $y$-intercept to be 0 (leaving the $y$-intercept unconstrained results in a model with almost identical slope and $y$-intercept of -30; it is indistinguishable from our selected model by the naked eye).
\[\sigma^2 = 0.0555327n^2\]
with confidence $p<0.0001$. Similarly, to study sample kurtoses (otherwise known as the fourth moment over the square of the second, which gives a measure of peakedness of the distribution), we fit the fourth sample moments to square of sample variances with a linear model and $y$-intercept constrained to 0:
\[m_4 = 3.5090019\sigma^4,\]
again with $p<0.0001$, so the sample kurtosis is approximately 3.5, independent of $n$. That is, we may reasonably conjecture:

\begin{conj} The writhe of knots in an $n\times n$ grid diagram follows a distribution with trivial expected value and skewness and with variance and kurtosis given by:
\[Var(w) = \frac{1}{18}n^2 \qquad and \qquad \kappa(w) = 3.5.\]
\end{conj}

\begin{figure}
\begin{subfigure}[b]{0.45\textwidth}
\includegraphics[width=\textwidth]{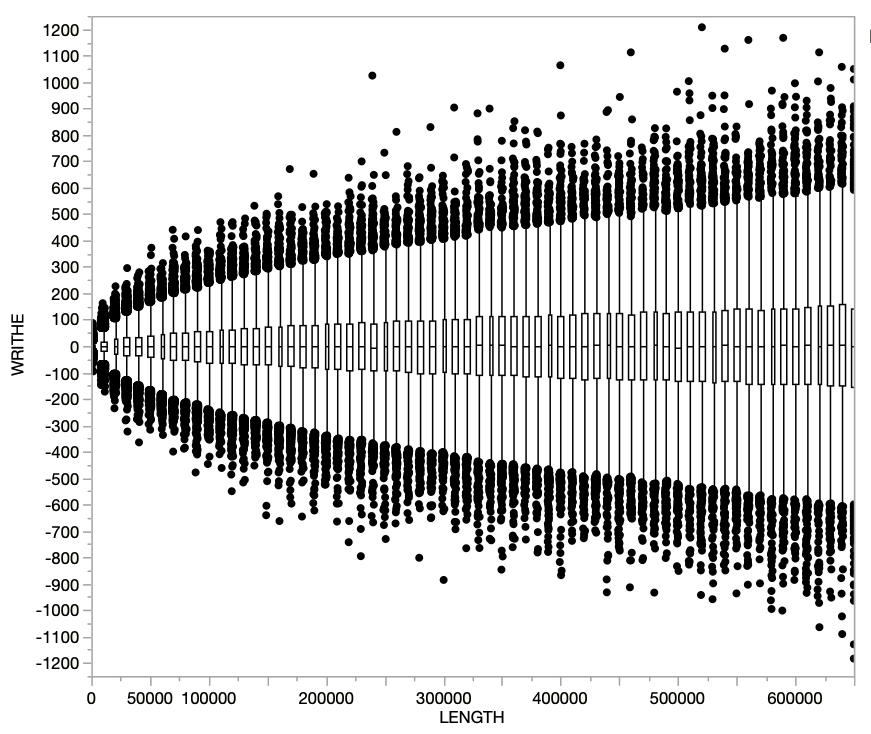}\caption{\raggedright }\label{fig:writhe_by_length_1}
\end{subfigure}
\begin{subfigure}[b]{0.45\textwidth}
\includegraphics[width=\textwidth]{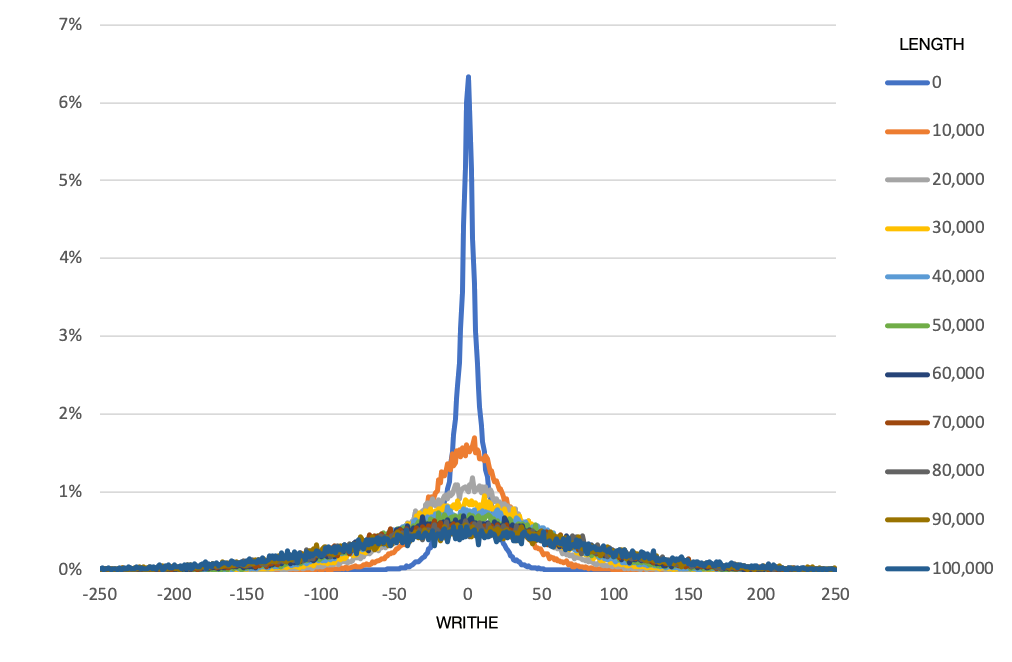}\caption{\raggedright }\label{fig:writhe_by_length_2}
\end{subfigure}
\caption{\raggedright Writhe vs knot length for randomly generated knot grid diagrams: (\subref{fig:writhe_by_length_1}) box plots for $l \leq 650,000$; (\subref{fig:writhe_by_length_2}) sample mass functions for $l \leq100,000$.}\label{fig:writhe_by_length}
\end{figure}

We also calculate the writhe-to-length ratio. We ran 1,000,000 cases randomly selected from $2 \leq n \leq 1000$. The longest knot had length just over 700,000, and knots on the upper end of this range were not well represented; we rounded length to the nearest 10,000 and restricted our set to those of length at most 650,000 (these range in size from 80,000 representatives for $l=0$ just under 7500 representatives for $l=650,000$). By the same reasoning used for $n$, in the general case, knots in grid diagrams come in pairs which are reflections and have opposite writhe, so
\[EV(w) = 0.\]

\begin{figure}
\begin{subfigure}[b]{0.45\textwidth}
\includegraphics[width=\textwidth]{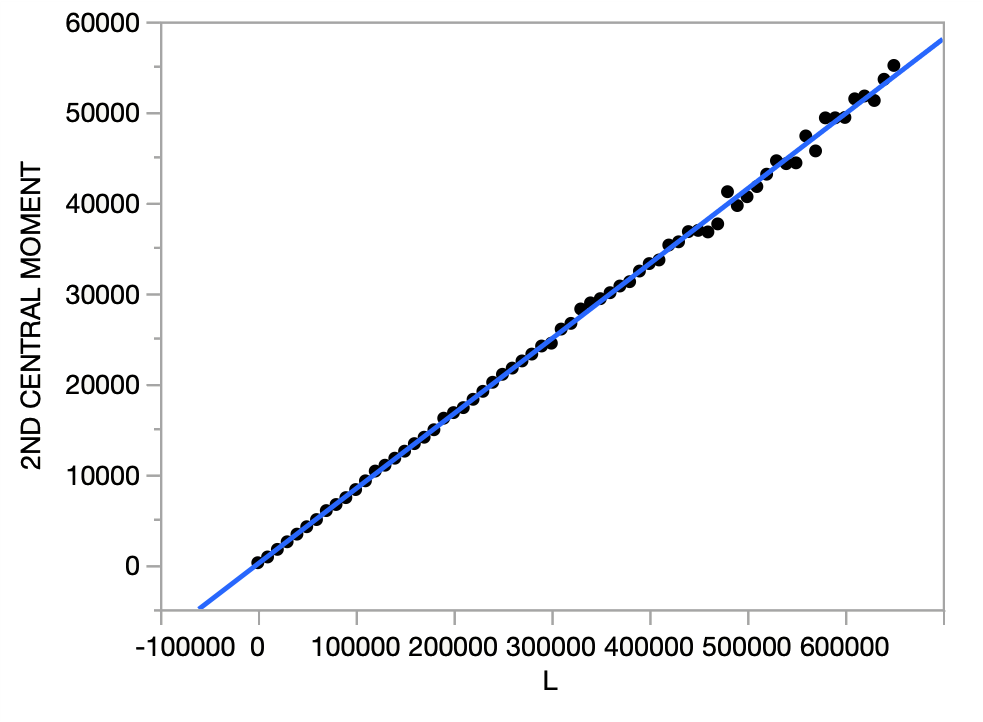}\caption{\raggedright }\label{fig:mu2_by_length}
\end{subfigure}
\begin{subfigure}[b]{0.45\textwidth}
\includegraphics[width=\textwidth]{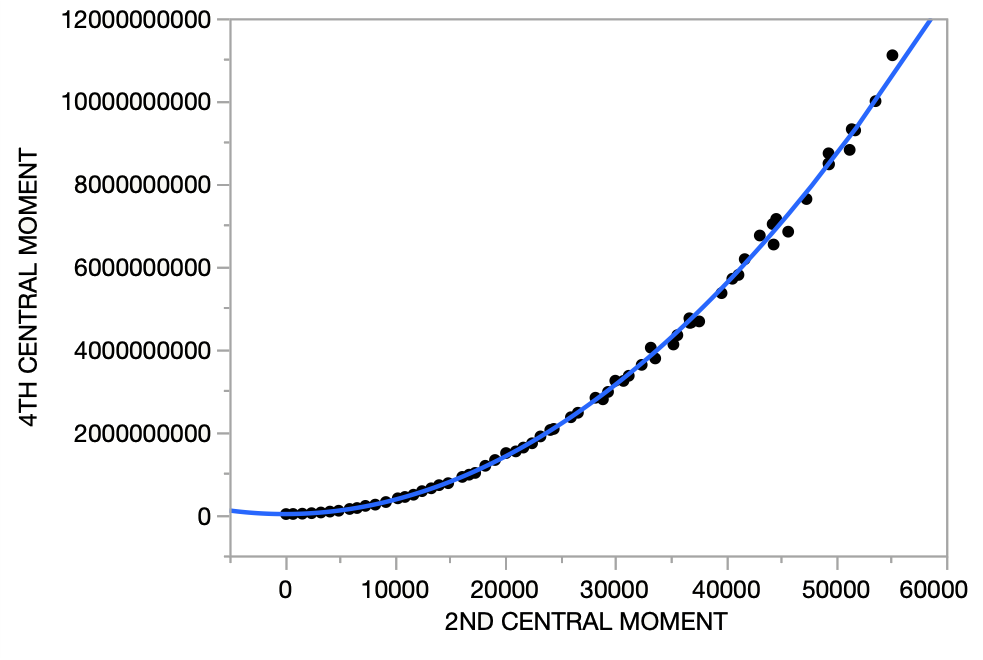}\caption{\raggedright }\label{fig:mu4_by_mu2_by_length}
\end{subfigure}
\caption{\raggedright Writhe vs knot length for randomly generated knot grid diagrams: (\subref{fig:mu2_by_length}) variance vs. length and (\subref{fig:mu4_by_mu2_by_length}) fourth moment vs. variance for $l \leq 650,000$.}\label{fig:writhe_by_length2}
\end{figure}

In this case, of the 65 cases sampled, only one ($l=0$, that is, knots of length 0-4999) does not contain 0 in its 95\% confidence interval. Figure~\ref{fig:writhe_by_length} demonstrates this, and it suggests a standard deviation proportional to $\sqrt l$. We again model the data (see Figure~\ref{fig:writhe_by_length2}) as
\[\sigma^2 = 0.0829184l\]
and
\[m_4 = 3.4763056\sigma^4,\]
suggesting a kurtosis around 3.48, both with $p < 0.0001$.

\subsection{Mathematical analysis}

The behavior of writhe on grid diagrams suggests it is susceptible to analysis via generating functions, but it was resistant to our efforts. We leave this project for future work.

\subsection{Other link models}

The numerical results above suggest writhe of random knots in a grid diagram displays behavior similar to writhe of random knots under the Petaluma model, where a knot may be described by a permutation. Even-Zohar \cite{even2017writhe} demonstrated a limit distribution for writhe; as with the grid diagrams, Petaluma knots come in pairs with opposite writhe (in this case, by reversing the permutation), so odd moments are trivial, but the variance tends to 2/3 and the kurtosis to 3.8.

\section{Generating functions}\label{sect:gf}

We gather here proofs of several results used in Section~\ref{sect:numb_comp}.

\begin{proof}[Proof of Theorem~\ref{thm:gf}] We begin by deriving formulae for $c_n$ and $c_{n,k}$. Consider an arbitrary link in an $n \times n$ grid. Let $e_i$ be the number of components that use $i$ rows and columns. We assume $e_1 = 0$ since this is a grid diagram, and $2e_2 + 3e_3 + \cdots = n$. To count the number of links which satisfy a given choice of component sizes $e_2, e_3, e_4, \cdots$, we must first divide the $n$ rows among the components: that is, we must divide them into subsets with $e_2$ of size $2$, $e_3$ of size $3$, and so on. There are $\frac{n!}{(2!)^{e_2}(3!)^{e_3}} \cdots$ ways to do this. Repeat for the rows. Next, since this was not a proper partition (we can distinguish between two subsets of the same size), we now divide by $e_2!e_3!\cdots$. Finally, for each subset  
\begin{multline*}c_n = \sum_{2e_2+3e_3+\cdots=n} \frac{n!n!(2!1!)^{e_2}(3!2!)^{e_3}\cdots}{((2!)^{e_2}(3!)^{e_3}\cdots)^2(e_2!e_3!\cdots)} \\= \sum_{2e_2+3e_3+\cdots=n} \frac{n!n!}{2^{e_2}3^{e_3}\cdots} \cdot \frac{1}{e_2!e_3!\cdots}\end{multline*}
%preprint\[c_n = \sum_{2e_2+3e_3+\cdots=n} \frac{n!n!(2!1!)^{e_2}(3!2!)^{e_3}\cdots}{((2!)^{e_2}(3!)^{e_3}\cdots)^2(e_2!e_3!\cdots)}= \sum_{2e_2+3e_3+\cdots=n} \frac{n!n!}{2^{e_2}3^{e_3}\cdots} \cdot \frac{1}{e_2!e_3!\cdots}\]

To find the number of links with exactly $k$ components $c_{n,k},$ we need merely add the condition that $e_2+e_3+\cdots = k$:
\[c_{n,k} = \sum_{\substack{2e_2+3e_3+\cdots=n \\ e_2+e_3+\cdots = k}} \frac{n!n!}{2^{e_2}3^{e_3}\cdots} \cdot \frac{1}{e_2!e_3!\cdots}\]

These formulas are not particularly easy to manipulate on their own, so we employ generating functions to make them manageable. Herbert Wilf wrote, ``A generating function is a clothesline on which we hang up a sequence of numbers for display.'' Recall that an ordinary generating function for $\overline c_n$ is some $g(x)$ with MacLaurin series $g(x) = \sum_{n = 0}^\infty \overline c_n x^n$, and so $\overline c_n =\frac{1}{n!}\frac{\partial^ng}{\partial x^n}(0)$. See, e.g., Tucker\cite[Ch. 6]{tucker} or Wilf\cite{wilf}. For $\overline c_n$ above, consider the function 
\[\displaystyle g(x) = \left(1 + \frac{x^2}{2}+ \frac{\left(\frac{x^2}{2}\right)^2}{2!}+ \cdots\right) \left(1 + \frac{x^3}{3} +\frac{\left(\frac{x^3}{3}\right)^2}{2!} + \cdots\right) \cdots \]
%\[\displaystyle g(x) = \left(1 + \frac{x^2}{2}+ \frac{\left(\frac{x^2}{2}\right)^2}{2!} + \frac{\left(\frac{x^2}{2}\right)^3}{3!} \cdots\right) \left(1 + \frac{x^3}{3} +\frac{\left(\frac{x^3}{3}\right)^2}{2!} + \frac{\left(\frac{x^3}{3}\right)^3}{3!}  \cdots\right) \cdots \]
We may think of $e_2$ as telling us which term to select from the first parentheses, then $e_3$ which term to select from the next, and so on. The power series expansion has as its degree $n$ term the sum of all formal products 
\[\left(\frac{x^2}{2}\right)^{e_2}\left(\frac{x^3}{3}\right)^{e_3} \cdots = \frac{x^{2e_2+3e_3+\cdots}}{2^{e_2}3^{e_3} \cdots} = \frac{x^n}{2^{e_2}3^{e_3} \cdots}\]
because $2e_2 + 3e_3 + \cdots = n$; this sum happens to be the same as $n!n! c_n x^n= \overline c_nx^n$. That is, 
\[g(x) = \sum_{n=0}^\infty \overline c_n x^n.\]
The expansion of $g(x)$ above is a little awkward; we may rewrite it as  
\[\displaystyle g(x) = e^\frac{x^2}{2}e^\frac{x^3}{3}\cdots = e^{\frac{x^2}{2}+\frac{x^3}{3}+\cdots}.\]
Additionally, 
\begin{multline*}\frac{\partial }{\partial x}\left[\frac{x^2}{2}+\frac{x^3}{3}+\cdots\right] = x+x^2+\cdots \\= \frac{1}{1-x}-1 = \frac{\partial }{\partial x}\left[-\ln|1-x|-x\right].\end{multline*}%preprint
%\[\frac{\partial }{\partial x}\left[\frac{x^2}{2}+\frac{x^3}{3}+\cdots\right] = x+x^2+\cdots= \frac{1}{1-x}-1 = \frac{\partial }{\partial x}\left[-\ln|1-x|-x\right].\] 
We restrict ourselves to the case of $x < 1$, and we perform a quick check at $x=0$ to verify that \mbox{$\frac{x^2}{2}+\frac{x^3}{3}+\cdots = -\ln|1-x|-x$}, so 
\[g(x) = e^{-\ln|1-x|-x}=(1-x)^{-1}e^{-x}.\]

We next need a multivariable generating function for $\overline c_{n,k}$. Let
%\[\displaystyle G(x,y) = \left(1 + \frac{x^2y}{2}+ \frac{\left(\frac{x^2y}{2}\right)^2}{2!} + \cdots\right) \left(1 + \frac{x^3y}{3} +\frac{\left(\frac{x^3y}{3}\right)^2}{2!} +  \cdots\right) \cdots \]
%\[\displaystyle G(x,y) = \left(1 + \frac{x^2y}{2}+ \frac{\left(\frac{x^2y}{2}\right)^2}{2!} + \frac{\left(\frac{x^2y}{2}\right)^3}{3!} \cdots\right) \left(1 + \frac{x^3y}{3} +\frac{\left(\frac{x^3y}{3}\right)^2}{2!} + \frac{\left(\frac{x^3y}{3}\right)^3}{3!}  \cdots\right) \cdots \]
\begin{multline*}
 G(x,y) =
 \left(1 + \frac{x^2y}{2}+ \frac{\left(\frac{x^2y}{2}\right)^2}{2!} + \cdots\right)\\
\cdot \left(1 + \frac{x^3y}{3} +\frac{\left(\frac{x^3y}{3}\right)^2}{2!} + \cdots\right) \cdots
\end{multline*}%preprint
which has as its degree $(n, k)$ term the sum of formal products
\[\left(\frac{1}{2}x^2y\right)^{e_2}\left(\frac{1}{3}x^3y\right)^{e_3} \cdots = \frac{x^ny^k}{2^{e_2}3^{e_3}\cdots}\]
for $2e_2 + 3e_3 + \cdots = n$ and $e_2 + e_3 \cdots = k$. This sum is again exactly $\overline c_{n,k}x^ny^k$, so $G(x,y)$ is a generating function for $\overline c_{n,k}$, or
\[G(x,y) = \sum_{n,k=0}^\infty \overline c_{n,k} x^ny^k.\]
We may also write it:
\[G(x,y) = e^{\left(\frac{x^2}{2}+\frac{x^3}{3}+\cdots\right)y} = \left(g(x)\right)^y = (1-x)^{-y}e^{-xy},\]
and of course $g(x) = G(x,1)$. 

We observe that we could fix $n$ and find a generating function for $\overline c_{n,k}$ as $k$ varies, 
\[f(y) = \sum_{k=0}^\infty \overline c_{n,k}y^k =  \frac{1}{n!}\frac{\partial ^n G}{\partial x^n}(0,y).\]
In other words,
\begin{multline*}
f(y) = \frac{1}{n!}\frac{\partial ^n}{\partial x^n} \left[(1-x)^{-y}e^{-xy}\right]_{x=0} \\
=\frac{1}{n!}\sum_{i=0}^n{n \choose i} \frac{\partial^{i}}{\partial x^{i}}[e^{-xy}]_{x=0}  \frac{\partial^{n-i}}{\partial x^{n-i}}\left[(1-x)^{-y}\right]_{x=0}\\
= \frac{1}{n!}\sum_{i=0}^n \frac{n!}{i!(n-i)!} \left[(-y)^ie^{-xy}\right]_{x=0} \left[\frac{(y+n-i-1)!}{(y-1)!(1-x)^{y+n-i}}\right]_{x=0}\\
=\sum_{i=0}^n (-1)^i\frac{(y+n-i-1)!}{i!(n-i)!(y-1)!} y^i
 = \sum_{i=0}^n(-1)^i{y+n-i-1 \choose y-1}\frac{y^i}{i!}.
\end{multline*}

On the other hand, we could also fix $k$ and find a generating function for $\overline c_{n,k}$ as $n$ varies. This function should take the form:
\begin{multline*}
h(x) = \sum_{k=0}^\infty \overline c_{n,k} x^n = \frac{1}{k!}\frac{\partial^k G}{\partial y^k}(x,0)
  = \frac{1}{k!}\frac{\partial^k}{\partial y^k}\left[\left((1-x)^{-1}e^{-x}\right)^y\right]_{y=0}\\
= \frac{1}{k!}\left[\left((1-x)^{-1}e^{-x}\right)^y\left(\ln\left((1-x)^{-1}e^{-x}\right)\right)^k\right]_{y=0}\\
=\frac{\left(-\ln(1-x)-x\right)^k}{k!} = (-1)^k\frac{\left(\ln(1-x)+x\right)^k}{k!}.
\end{multline*}
\end{proof}

\begin{proof}[Proof of Theorem~\ref{thm:cnk}]
Assume $n>1$. We will actually verify the bound 
\[\overline c_{n,k} \leq \frac{(\log_2 n)^{k-1}}{n}.\]
This bound is sharp for $k=1$ since $\overline c_{n,1} = \frac{1}{n}$ by Theorem~\ref{thm:cn}. 

For larger $k$, we use the generating function $h(x)$ from Theorem~\ref{thm:gf}:
\begin{multline*}
\overline c_{n,k} = \frac{1}{n!}\frac{\partial^n}{\partial x^n}\left[\frac{1}{k!}(-\ln(1-x)-x)^k\right]_{x=0}\\ %preprint
 = \frac{1}{n!}\frac{\partial^{n-1}}{\partial x^{n-1}}\left[\frac{1}{(k-1)!}\left(-\ln(1-x)-x\right)^{k-1} \left(\frac{1}{1-x}-1\right)\right]_{x=0}\\
=\frac{1}{n!}\sum_{i=0}^{n-1} {n-1 \choose i}\frac{\partial^i}{\partial x^i}\left[\frac{\left(-\ln(1-x)-x\right)^{k-1}}{(k-1)!}\right]_{x=0}\\ %preprint
\cdot \frac{\partial^{n-i-1}}{\partial x^{n-i-1}}\left[\frac{1}{1-x}-1\right]_{x=0}\\
=\frac{1}{n!}\sum_{i=1}^{n-2}\frac{(n-1)!}{i!(n-i-1)!}  i!\overline c_{i,k-1} (n-i-1)!
=\frac{1}{n}\sum_{i=1}^{n-2} \overline c_{i,k-1}
\end{multline*}
which, inducting on $k$, gives
\[\overline c_{n,k} \leq \frac{1}{n}\sum_{i=1}^{n-2} \frac{(\log_2 i)^{k-2}}{i}
\leq \frac{1}{n}(\log_2 n)^{k-2}\sum_{i=1}^{n-1} \frac{1}{i} \leq \frac{1}{n}(\log_2 n)^{k-1}\]
where the last inequality comes from the classical harmonic bound $H_{n-1} \leq \log_2 n$.
\end{proof}

\begin{proof}[Proof of Corollary~\ref{cor:fwdconj}]
Observe that 
\[ \lim_{n\rightarrow \infty}\frac{(\log_2 n)^{k-1}}{n} = 0\]
by a repeated application of l'H\^opital's rule, so $\overline c_{n,k} \rightarrow 0$ as well. Since $\overline c_n \rightarrow \frac{1}{e}$ by Theorem~\ref{thm:cn}, 
\[\lim_{n\rightarrow \infty} \frac{c_{n,k}}{c_n} = \lim_{n\rightarrow \infty} \frac{\overline c_{n,k}}{\overline c_n} = 0.\]
In other words, $\Pr (k \mid n) \rightarrow 0$. The result for a fixed link type follows trivially: if all links of $k$ components become vanishingly rare, then the particular ones which represent a given link type become even rarer.
\end{proof}

\begin{proof}[Proof of Theorem~\ref{thm:last_thm}]
To study expected value, note that (after fixing $n$) the probability of a randomly selected link having exactly $k$ components is
\[p_k = \frac{c_{n,k}}{c_n} = \frac{\overline c_{n,k}}{\overline c_n}\]
which has probability generating function
\[\displaystyle P(y) = \sum_{k=0}^\infty p_ky^k =  \sum_{k=0}^\infty\frac{\overline c_{n,k}}{\overline c_n}y^k = \frac{1}{\overline c_n} f(y).\]
where $f(y)$ is the generating function for $\overline c_{n,k}$ (with $n$ fixed) from Theorem~\ref{thm:gf}. The expected value will be 
\[EV = P'(1) = \frac{1}{\overline c_n}\frac{\partial f}{\partial y}(1)\]
and the variance will be
\[Var = P''(1) - P'(1)\]

Before we proceed, note that, if $i < n$,
\begin{multline*}
\frac{\partial}{\partial y}\left[{y+n-i-1 \choose y-1}\right] = \frac{\partial}{\partial y}\left[\frac{1}{(n-i)!}\prod_{j=0}^{n-i-1}(y+j)\right]\\
= \frac{1}{(n-i)!}\sum_{k=0}^{n-i-1} \frac{\prod_{j=0}^{n-i-1}(y+j)}{y+k}\\ %preprint
= {y+n-i-1 \choose y-1}\sum_{k=0}^{n-i-1} \frac{1}{y+k}\end{multline*}
In particular, for $0 \leq i \leq n$ (if we assume the trivial case \mbox{$H_0 = 0$}),
\[\frac{\partial}{\partial y}\left[{y+n-i-1 \choose y-1}\right]_{y=1} = H_{n-i}\]
and
\begin{multline*}
\frac{\partial^2}{\partial y^2}\left[{y+n-i-1 \choose y-1}\right]_{y=1}\\ %preprint
 = \frac{\partial}{\partial y}\left[ {y+n-i-1 \choose y-1}\sum_{k=0}^{n-i-1} \frac{1}{y+k}\right]_{y=1}\\
 = \left[ {y+n-i-1 \choose y-1}\left(\left(\sum_{k=0}^{n-i-1} \frac{1}{y+k}\right)^2 -\sum_{k=0}^{n-i-1} \frac{1}{(y+k)^2}\right)\right]_{y=1}\\ %preprint
 =H_{n-i}^2-H_{n-i,2}
\end{multline*}

Therefore,
%preprint\[ EV(k) =\frac{1}{\overline c_n} \sum_{i=0}^n\frac{\partial}{\partial y}\left[(-1)^i{y+n-i-1 \choose y-1}\frac{y^i}{i!}\right]_{y=1}= \frac{1}{\overline c_n}\sum_{i=0}^{n}(-1)^i\frac{H_{n-i}+i}{i!}\]
\begin{multline*} 
EV(k) =\frac{1}{\overline c_n} \sum_{i=0}^n\frac{\partial}{\partial y}\left[(-1)^i{y+n-i-1 \choose y-1}\frac{y^i}{i!}\right]_{y=1}\\
= \frac{1}{\overline c_n}\sum_{i=0}^{n}(-1)^i\frac{H_{n-i}+i}{i!} =  \frac{1}{\overline c_n}\sum_{i=0}^{n-1}(-1)^i\frac{H_{n-i}-1}{i!}
\end{multline*}
and the variance obeys
\begin{multline*}
Var(k) + EV(k) = \frac{1}{\overline c_n} \sum_{i=0}^n\frac{\partial^2}{\partial y^2}\left[(-1)^i{y+n-i-1 \choose y-1}\frac{y^i}{i!}\right]_{y=1}\\
=\frac{1}{\overline c_n}\sum_{i=0}^{n}(-1)^i\frac{H^2_{n-i}-H_{n-i,2} + 2iH_{n-i} + i(i-1)}{i!}
\end{multline*}
so
\[Var(k) =\frac{1}{\overline c_n}\sum_{i=0}^{n}(-1)^i\frac{H^2_{n-i}-H_{n-i,2} + (2i-1)H_{n-i} + i(i-2)}{i!}\]

\end{proof}

\section{Future work}\label{sect:future_work}

The exploratory data analysis on writhe above begs to be verified mathematically. 

We are also involved in a future exploration\cite{doigtypicalii} of genus as a model of knot complexity; in particular, we investigate the effects of a crossing change and tangle change on genus.

\section{Acknowledgements}

Thanks to Creighton University for giving me an environment in which I am free to wander about the mathematical world and explore new things. Thanks also to my student Billy Duckworth, whose ideas during his summer research inspired me to start digging into the world of typical knots. Finally, thanks to to Chaim Even-Zohar, who provided several insightful suggestions for this manuscript, %including the recommendation that mean and SD in Section~\ref{sect:count_comp} be fit to $\sqrt(\ln n)$ rather than $\ln n$. 
and to the many other previous researchers referenced here whose work was inspiring and interesting as well as highly informative.

%\end{singlespacing}
%\bibliographystyle{alpha}
\bibliography{../../Documents/math/biblio-MASTER}

%ISSUE: OEIS won't display the url properly. 
\end{document}

%% file: knotsv2.pdf_t
\begin{picture}(0,0)%
\includegraphics{knotsv2.pdf}%
\end{picture}%
\setlength{\unitlength}{3947sp}%
\begingroup\makeatletter\ifx\SetFigFont\undefined%
\gdef\SetFigFont#1#2#3#4#5{%
  \reset@font\fontsize{#1}{#2pt}%
  \fontfamily{#3}\fontseries{#4}\fontshape{#5}%
  \selectfont}%
\fi\endgroup%
\begin{picture}(10908,3777)(547,-5284)
\end{picture}%

%% file: grid_diagram.pdf_t
\begin{picture}(0,0)%
\includegraphics{grid_diagram.pdf}%
\end{picture}%
\setlength{\unitlength}{3947sp}%
\begingroup\makeatletter\ifx\SetFigFont\undefined%
\gdef\SetFigFont#1#2#3#4#5{%
  \reset@font\fontsize{#1}{#2pt}%
  \fontfamily{#3}\fontseries{#4}\fontshape{#5}%
  \selectfont}%
\fi\endgroup%
\begin{picture}(8748,3544)(-14,-6572)
\put(2401,-6436){\makebox(0,0)[b]{\smash{{\SetFigFont{20}{24.0}{\rmdefault}{\mddefault}{\updefault}{\color[rgb]{0,0,0}$\kappa_4$}%
}}}}
\put(3001,-6436){\makebox(0,0)[b]{\smash{{\SetFigFont{20}{24.0}{\rmdefault}{\mddefault}{\updefault}{\color[rgb]{0,0,0}$\kappa_5$}%
}}}}
\put(601,-6436){\makebox(0,0)[b]{\smash{{\SetFigFont{20}{24.0}{\rmdefault}{\mddefault}{\updefault}{\color[rgb]{0,0,0}$\kappa_1$}%
}}}}
\put(1201,-6436){\makebox(0,0)[b]{\smash{{\SetFigFont{20}{24.0}{\rmdefault}{\mddefault}{\updefault}{\color[rgb]{0,0,0}$\kappa_2$}%
}}}}
\put(1801,-6436){\makebox(0,0)[b]{\smash{{\SetFigFont{20}{24.0}{\rmdefault}{\mddefault}{\updefault}{\color[rgb]{0,0,0}$\kappa_3$}%
}}}}
\put(  1,-3436){\makebox(0,0)[b]{\smash{{\SetFigFont{20}{24.0}{\rmdefault}{\mddefault}{\updefault}{\color[rgb]{0,0,0}$\rho_5$}%
}}}}
\put(  1,-4036){\makebox(0,0)[b]{\smash{{\SetFigFont{20}{24.0}{\rmdefault}{\mddefault}{\updefault}{\color[rgb]{0,0,0}$\rho_4$}%
}}}}
\put(  1,-4636){\makebox(0,0)[b]{\smash{{\SetFigFont{20}{24.0}{\rmdefault}{\mddefault}{\updefault}{\color[rgb]{0,0,0}$\rho_3$}%
}}}}
\put(  1,-5236){\makebox(0,0)[b]{\smash{{\SetFigFont{20}{24.0}{\rmdefault}{\mddefault}{\updefault}{\color[rgb]{0,0,0}$\rho_2$}%
}}}}
\put(  1,-5836){\makebox(0,0)[b]{\smash{{\SetFigFont{20}{24.0}{\rmdefault}{\mddefault}{\updefault}{\color[rgb]{0,0,0}$\rho_1$}%
}}}}
\put(7801,-6436){\makebox(0,0)[b]{\smash{{\SetFigFont{20}{24.0}{\rmdefault}{\mddefault}{\updefault}{\color[rgb]{0,0,0}$\kappa_4$}%
}}}}
\put(8401,-6436){\makebox(0,0)[b]{\smash{{\SetFigFont{20}{24.0}{\rmdefault}{\mddefault}{\updefault}{\color[rgb]{0,0,0}$\kappa_5$}%
}}}}
\put(6001,-6436){\makebox(0,0)[b]{\smash{{\SetFigFont{20}{24.0}{\rmdefault}{\mddefault}{\updefault}{\color[rgb]{0,0,0}$\kappa_1$}%
}}}}
\put(6601,-6436){\makebox(0,0)[b]{\smash{{\SetFigFont{20}{24.0}{\rmdefault}{\mddefault}{\updefault}{\color[rgb]{0,0,0}$\kappa_2$}%
}}}}
\put(7201,-6436){\makebox(0,0)[b]{\smash{{\SetFigFont{20}{24.0}{\rmdefault}{\mddefault}{\updefault}{\color[rgb]{0,0,0}$\kappa_3$}%
}}}}
\put(5401,-3436){\makebox(0,0)[b]{\smash{{\SetFigFont{20}{24.0}{\rmdefault}{\mddefault}{\updefault}{\color[rgb]{0,0,0}$\rho_5$}%
}}}}
\put(5401,-4036){\makebox(0,0)[b]{\smash{{\SetFigFont{20}{24.0}{\rmdefault}{\mddefault}{\updefault}{\color[rgb]{0,0,0}$\rho_4$}%
}}}}
\put(5401,-4636){\makebox(0,0)[b]{\smash{{\SetFigFont{20}{24.0}{\rmdefault}{\mddefault}{\updefault}{\color[rgb]{0,0,0}$\rho_3$}%
}}}}
\put(5401,-5236){\makebox(0,0)[b]{\smash{{\SetFigFont{20}{24.0}{\rmdefault}{\mddefault}{\updefault}{\color[rgb]{0,0,0}$\rho_2$}%
}}}}
\put(5401,-5836){\makebox(0,0)[b]{\smash{{\SetFigFont{20}{24.0}{\rmdefault}{\mddefault}{\updefault}{\color[rgb]{0,0,0}$\rho_1$}%
}}}}
\end{picture}%

%% file: v7-ARXIV.bbl
\begin{thebibliography}{48}%
\makeatletter
\providecommand \@ifxundefined [1]{%
 \@ifx{#1\undefined}
}%
\providecommand \@ifnum [1]{%
 \ifnum #1\expandafter \@firstoftwo
 \else \expandafter \@secondoftwo
 \fi
}%
\providecommand \@ifx [1]{%
 \ifx #1\expandafter \@firstoftwo
 \else \expandafter \@secondoftwo
 \fi
}%
\providecommand \natexlab [1]{#1}%
\providecommand \enquote  [1]{``#1''}%
\providecommand \bibnamefont  [1]{#1}%
\providecommand \bibfnamefont [1]{#1}%
\providecommand \citenamefont [1]{#1}%
\providecommand \href@noop [0]{\@secondoftwo}%
\providecommand \href [0]{\begingroup \@sanitize@url \@href}%
\providecommand \@href[1]{\@@startlink{#1}\@@href}%
\providecommand \@@href[1]{\endgroup#1\@@endlink}%
\providecommand \@sanitize@url [0]{\catcode `\\12\catcode `\$12\catcode
  `\&12\catcode `\#12\catcode `\^12\catcode `\_12\catcode `\%12\relax}%
\providecommand \@@startlink[1]{}%
\providecommand \@@endlink[0]{}%
\providecommand \url  [0]{\begingroup\@sanitize@url \@url }%
\providecommand \@url [1]{\endgroup\@href {#1}{\urlprefix }}%
\providecommand \urlprefix  [0]{URL }%
\providecommand \Eprint [0]{\href }%
\providecommand \doibase [0]{http://dx.doi.org/}%
\providecommand \selectlanguage [0]{\@gobble}%
\providecommand \bibinfo  [0]{\@secondoftwo}%
\providecommand \bibfield  [0]{\@secondoftwo}%
\providecommand \translation [1]{[#1]}%
\providecommand \BibitemOpen [0]{}%
\providecommand \bibitemStop [0]{}%
\providecommand \bibitemNoStop [0]{.\EOS\space}%
\providecommand \EOS [0]{\spacefactor3000\relax}%
\providecommand \BibitemShut  [1]{\csname bibitem#1\endcsname}%
\let\auto@bib@innerbib\@empty
%</preamble>
\bibitem [{\citenamefont {Doig}()}]{doigtypicalii}%
  \BibitemOpen
  \bibfield  {author} {\bibinfo {author} {\bibfnamefont {M.~I.}\ \bibnamefont
  {Doig}},\ }\href@noop {} {\enquote {\bibinfo {title} {Typical links: The
  effects of crossing change on genus},}\ }\bibinfo {note} {In
  progress}\BibitemShut {NoStop}%
\bibitem [{\citenamefont {Farhi}\ \emph {et~al.}(2012)\citenamefont {Farhi},
  \citenamefont {Gosset}, \citenamefont {Hassidim}, \citenamefont
  {Lutomirski},\ and\ \citenamefont {Shor}}]{knotcrypto}%
  \BibitemOpen
  \bibfield  {author} {\bibinfo {author} {\bibfnamefont {E.}~\bibnamefont
  {Farhi}}, \bibinfo {author} {\bibfnamefont {D.}~\bibnamefont {Gosset}},
  \bibinfo {author} {\bibfnamefont {A.}~\bibnamefont {Hassidim}}, \bibinfo
  {author} {\bibfnamefont {A.}~\bibnamefont {Lutomirski}}, \ and\ \bibinfo
  {author} {\bibfnamefont {P.}~\bibnamefont {Shor}},\ }\bibfield  {title}
  {\enquote {\bibinfo {title} {Quantum money from knots},}\ }in\ \href@noop {}
  {\emph {\bibinfo {booktitle} {Proceedings of the 3rd Innovations in
  Theoretical Computer Science Conference}}}\ (\bibinfo {year} {2012})\ pp.\
  \bibinfo {pages} {276--289}\BibitemShut {NoStop}%
\bibitem [{\citenamefont {Even-Zohar}(2017{\natexlab{a}})}]{evenzoharmodels}%
  \BibitemOpen
  \bibfield  {author} {\bibinfo {author} {\bibfnamefont {C.}~\bibnamefont
  {Even-Zohar}},\ }\bibfield  {title} {\enquote {\bibinfo {title} {Models of
  random knots},}\ }\href@noop {} {\bibfield  {journal} {\bibinfo  {journal}
  {Journal of Applied and Computational Topology}\ }\textbf {\bibinfo {volume}
  {1}},\ \bibinfo {pages} {263--296} (\bibinfo {year}
  {2017}{\natexlab{a}})}\BibitemShut {NoStop}%
\bibitem [{\citenamefont {Caraglio}, \citenamefont {Micheletti},\ and\
  \citenamefont {Orlandini}(2015)}]{caraglio}%
  \BibitemOpen
  \bibfield  {author} {\bibinfo {author} {\bibfnamefont {M.}~\bibnamefont
  {Caraglio}}, \bibinfo {author} {\bibfnamefont {C.}~\bibnamefont
  {Micheletti}}, \ and\ \bibinfo {author} {\bibfnamefont {E.}~\bibnamefont
  {Orlandini}},\ }\bibfield  {title} {\enquote {\bibinfo {title} {Stretching
  response of knotted and unknotted polymer chains},}\ }\href@noop {}
  {\bibfield  {journal} {\bibinfo  {journal} {Physical review letters}\
  }\textbf {\bibinfo {volume} {115}},\ \bibinfo {pages} {188301} (\bibinfo
  {year} {2015})}\BibitemShut {NoStop}%
\bibitem [{\citenamefont {Plesa}\ \emph {et~al.}(2016)\citenamefont {Plesa},
  \citenamefont {Verschueren}, \citenamefont {Pud}, \citenamefont {van~der
  Torre}, \citenamefont {Ruitenberg}, \citenamefont {Witteveen}, \citenamefont
  {Jonsson}, \citenamefont {Grosberg}, \citenamefont {Rabin},\ and\
  \citenamefont {Dekker}}]{plesa2016direct}%
  \BibitemOpen
  \bibfield  {author} {\bibinfo {author} {\bibfnamefont {C.}~\bibnamefont
  {Plesa}}, \bibinfo {author} {\bibfnamefont {D.}~\bibnamefont {Verschueren}},
  \bibinfo {author} {\bibfnamefont {S.}~\bibnamefont {Pud}}, \bibinfo {author}
  {\bibfnamefont {J.}~\bibnamefont {van~der Torre}}, \bibinfo {author}
  {\bibfnamefont {J.~W.}\ \bibnamefont {Ruitenberg}}, \bibinfo {author}
  {\bibfnamefont {M.~J.}\ \bibnamefont {Witteveen}}, \bibinfo {author}
  {\bibfnamefont {M.~P.}\ \bibnamefont {Jonsson}}, \bibinfo {author}
  {\bibfnamefont {A.~Y.}\ \bibnamefont {Grosberg}}, \bibinfo {author}
  {\bibfnamefont {Y.}~\bibnamefont {Rabin}}, \ and\ \bibinfo {author}
  {\bibfnamefont {C.}~\bibnamefont {Dekker}},\ }\bibfield  {title} {\enquote
  {\bibinfo {title} {Direct observation of {DNA} knots using a solid-state
  nanopore},}\ }\href@noop {} {\bibfield  {journal} {\bibinfo  {journal}
  {Nature nanotechnology}\ }\textbf {\bibinfo {volume} {11}},\ \bibinfo {pages}
  {1093} (\bibinfo {year} {2016})}\BibitemShut {NoStop}%
\bibitem [{\citenamefont {Arsuaga}\ \emph {et~al.}(2002)\citenamefont
  {Arsuaga}, \citenamefont {V{\'a}zquez}, \citenamefont {Trigueros},
  \citenamefont {Roca} \emph {et~al.}}]{arsuaga2002knotting}%
  \BibitemOpen
  \bibfield  {author} {\bibinfo {author} {\bibfnamefont {J.}~\bibnamefont
  {Arsuaga}}, \bibinfo {author} {\bibfnamefont {M.}~\bibnamefont
  {V{\'a}zquez}}, \bibinfo {author} {\bibfnamefont {S.}~\bibnamefont
  {Trigueros}}, \bibinfo {author} {\bibfnamefont {J.}~\bibnamefont {Roca}},
  \emph {et~al.},\ }\bibfield  {title} {\enquote {\bibinfo {title} {Knotting
  probability of {DNA} molecules confined in restricted volumes: {DNA} knotting
  in phage capsids},}\ }\href@noop {} {\bibfield  {journal} {\bibinfo
  {journal} {Proceedings of the National Academy of Sciences}\ }\textbf
  {\bibinfo {volume} {99}},\ \bibinfo {pages} {5373--5377} (\bibinfo {year}
  {2002})}\BibitemShut {NoStop}%
\bibitem [{\citenamefont {Cebri{\'a}n}\ \emph {et~al.}(2015)\citenamefont
  {Cebri{\'a}n}, \citenamefont {Kadomatsu-Hermosa}, \citenamefont {Cast{\'a}n},
  \citenamefont {Mart{\'\i}nez}, \citenamefont {Parra}, \citenamefont
  {Fern{\'a}ndez-Nestosa}, \citenamefont {Schaerer}, \citenamefont
  {Mart{\'\i}nez-Robles}, \citenamefont {Hern{\'a}ndez}, \citenamefont {Krimer}
  \emph {et~al.}}]{cebrian2015electrophoretic}%
  \BibitemOpen
  \bibfield  {author} {\bibinfo {author} {\bibfnamefont {J.}~\bibnamefont
  {Cebri{\'a}n}}, \bibinfo {author} {\bibfnamefont {M.~J.}\ \bibnamefont
  {Kadomatsu-Hermosa}}, \bibinfo {author} {\bibfnamefont {A.}~\bibnamefont
  {Cast{\'a}n}}, \bibinfo {author} {\bibfnamefont {V.}~\bibnamefont
  {Mart{\'\i}nez}}, \bibinfo {author} {\bibfnamefont {C.}~\bibnamefont
  {Parra}}, \bibinfo {author} {\bibfnamefont {M.~J.}\ \bibnamefont
  {Fern{\'a}ndez-Nestosa}}, \bibinfo {author} {\bibfnamefont {C.}~\bibnamefont
  {Schaerer}}, \bibinfo {author} {\bibfnamefont {M.-L.}\ \bibnamefont
  {Mart{\'\i}nez-Robles}}, \bibinfo {author} {\bibfnamefont {P.}~\bibnamefont
  {Hern{\'a}ndez}}, \bibinfo {author} {\bibfnamefont {D.~B.}\ \bibnamefont
  {Krimer}},  \emph {et~al.},\ }\bibfield  {title} {\enquote {\bibinfo {title}
  {Electrophoretic mobility of supercoiled, catenated and knotted {DNA}
  molecules},}\ }\href@noop {} {\bibfield  {journal} {\bibinfo  {journal}
  {Nucleic acids research}\ }\textbf {\bibinfo {volume} {43}},\ \bibinfo
  {pages} {e24--e24} (\bibinfo {year} {2015})}\BibitemShut {NoStop}%
\bibitem [{\citenamefont {Stasiak}\ \emph {et~al.}(1996)\citenamefont
  {Stasiak}, \citenamefont {Katritch}, \citenamefont {Bednar}, \citenamefont
  {Michoud},\ and\ \citenamefont {Dubochet}}]{stasiak1996electrophoretic}%
  \BibitemOpen
  \bibfield  {author} {\bibinfo {author} {\bibfnamefont {A.}~\bibnamefont
  {Stasiak}}, \bibinfo {author} {\bibfnamefont {V.}~\bibnamefont {Katritch}},
  \bibinfo {author} {\bibfnamefont {J.}~\bibnamefont {Bednar}}, \bibinfo
  {author} {\bibfnamefont {D.}~\bibnamefont {Michoud}}, \ and\ \bibinfo
  {author} {\bibfnamefont {J.}~\bibnamefont {Dubochet}},\ }\bibfield  {title}
  {\enquote {\bibinfo {title} {Electrophoretic mobility of {DNA} knots},}\
  }\href@noop {} {\bibfield  {journal} {\bibinfo  {journal} {Nature}\ }\textbf
  {\bibinfo {volume} {384}},\ \bibinfo {pages} {122--122} (\bibinfo {year}
  {1996})}\BibitemShut {NoStop}%
\bibitem [{\citenamefont {Olavarrieta}\ \emph {et~al.}(2002)\citenamefont
  {Olavarrieta}, \citenamefont {Mart{\'\i}nez-Robles}, \citenamefont {Sogo},
  \citenamefont {Stasiak}, \citenamefont {Hernandez}, \citenamefont {Krimer},\
  and\ \citenamefont {Schvartzman}}]{olavarrieta2002supercoiling}%
  \BibitemOpen
  \bibfield  {author} {\bibinfo {author} {\bibfnamefont {L.}~\bibnamefont
  {Olavarrieta}}, \bibinfo {author} {\bibfnamefont {M.~L.}\ \bibnamefont
  {Mart{\'\i}nez-Robles}}, \bibinfo {author} {\bibfnamefont {J.}~\bibnamefont
  {Sogo}}, \bibinfo {author} {\bibfnamefont {A.}~\bibnamefont {Stasiak}},
  \bibinfo {author} {\bibfnamefont {P.}~\bibnamefont {Hernandez}}, \bibinfo
  {author} {\bibfnamefont {D.~B.}\ \bibnamefont {Krimer}}, \ and\ \bibinfo
  {author} {\bibfnamefont {J.~B.}\ \bibnamefont {Schvartzman}},\ }\bibfield
  {title} {\enquote {\bibinfo {title} {Supercoiling, knotting and replication
  fork reversal in partially replicated plasmids},}\ }\href@noop {} {\bibfield
  {journal} {\bibinfo  {journal} {Nucleic acids research}\ }\textbf {\bibinfo
  {volume} {30}},\ \bibinfo {pages} {656--666} (\bibinfo {year}
  {2002})}\BibitemShut {NoStop}%
\bibitem [{\citenamefont {Trigueros}\ \emph {et~al.}(2001)\citenamefont
  {Trigueros}, \citenamefont {Arsuaga}, \citenamefont {Vazquez}, \citenamefont
  {Sumners},\ and\ \citenamefont {Roca}}]{trigueros2001novel}%
  \BibitemOpen
  \bibfield  {author} {\bibinfo {author} {\bibfnamefont {S.}~\bibnamefont
  {Trigueros}}, \bibinfo {author} {\bibfnamefont {J.}~\bibnamefont {Arsuaga}},
  \bibinfo {author} {\bibfnamefont {M.~E.}\ \bibnamefont {Vazquez}}, \bibinfo
  {author} {\bibfnamefont {D.~W.}\ \bibnamefont {Sumners}}, \ and\ \bibinfo
  {author} {\bibfnamefont {J.}~\bibnamefont {Roca}},\ }\bibfield  {title}
  {\enquote {\bibinfo {title} {Novel display of knotted {DNA} molecules by
  two-dimensional gel electrophoresis},}\ }\href@noop {} {\bibfield  {journal}
  {\bibinfo  {journal} {Nucleic acids research}\ }\textbf {\bibinfo {volume}
  {29}},\ \bibinfo {pages} {e67--e67} (\bibinfo {year} {2001})}\BibitemShut
  {NoStop}%
\bibitem [{\citenamefont {Trigueros}\ and\ \citenamefont
  {Roca}(2007)}]{trigueros2007production}%
  \BibitemOpen
  \bibfield  {author} {\bibinfo {author} {\bibfnamefont {S.}~\bibnamefont
  {Trigueros}}\ and\ \bibinfo {author} {\bibfnamefont {J.}~\bibnamefont
  {Roca}},\ }\bibfield  {title} {\enquote {\bibinfo {title} {Production of
  highly knotted {DNA} by means of cosmid circularization inside phage
  capsids},}\ }\href@noop {} {\bibfield  {journal} {\bibinfo  {journal} {BMC
  biotechnology}\ }\textbf {\bibinfo {volume} {7}},\ \bibinfo {pages} {94}
  (\bibinfo {year} {2007})}\BibitemShut {NoStop}%
\bibitem [{\citenamefont {Weber}\ \emph
  {et~al.}(2006{\natexlab{a}})\citenamefont {Weber}, \citenamefont {Stasiak},
  \citenamefont {De~Los~Rios},\ and\ \citenamefont
  {Dietler}}]{weber2006numerical}%
  \BibitemOpen
  \bibfield  {author} {\bibinfo {author} {\bibfnamefont {C.}~\bibnamefont
  {Weber}}, \bibinfo {author} {\bibfnamefont {A.}~\bibnamefont {Stasiak}},
  \bibinfo {author} {\bibfnamefont {P.}~\bibnamefont {De~Los~Rios}}, \ and\
  \bibinfo {author} {\bibfnamefont {G.}~\bibnamefont {Dietler}},\ }\bibfield
  {title} {\enquote {\bibinfo {title} {Numerical simulation of gel
  electrophoresis of {DNA} knots in weak and strong electric fields},}\
  }\href@noop {} {\bibfield  {journal} {\bibinfo  {journal} {Biophysical
  journal}\ }\textbf {\bibinfo {volume} {90}},\ \bibinfo {pages} {3100--3105}
  (\bibinfo {year} {2006}{\natexlab{a}})}\BibitemShut {NoStop}%
\bibitem [{\citenamefont {Weber}\ \emph
  {et~al.}(2006{\natexlab{b}})\citenamefont {Weber}, \citenamefont
  {De~Los~Rios}, \citenamefont {Dietler},\ and\ \citenamefont
  {Stasiak}}]{weber2006simulations}%
  \BibitemOpen
  \bibfield  {author} {\bibinfo {author} {\bibfnamefont {C.}~\bibnamefont
  {Weber}}, \bibinfo {author} {\bibfnamefont {P.}~\bibnamefont {De~Los~Rios}},
  \bibinfo {author} {\bibfnamefont {G.}~\bibnamefont {Dietler}}, \ and\
  \bibinfo {author} {\bibfnamefont {A.}~\bibnamefont {Stasiak}},\ }\bibfield
  {title} {\enquote {\bibinfo {title} {Simulations of electrophoretic
  collisions of {DNA} knots with gel obstacles},}\ }\href@noop {} {\bibfield
  {journal} {\bibinfo  {journal} {Journal of Physics: Condensed Matter}\
  }\textbf {\bibinfo {volume} {18}},\ \bibinfo {pages} {S161} (\bibinfo {year}
  {2006}{\natexlab{b}})}\BibitemShut {NoStop}%
\bibitem [{\citenamefont {Orlandini}(2017)}]{orlandini2017statics}%
  \BibitemOpen
  \bibfield  {author} {\bibinfo {author} {\bibfnamefont {E.}~\bibnamefont
  {Orlandini}},\ }\bibfield  {title} {\enquote {\bibinfo {title} {Statics and
  dynamics of {DNA} knotting},}\ }\href@noop {} {\bibfield  {journal} {\bibinfo
   {journal} {Journal of Physics A: Mathematical and Theoretical}\ }\textbf
  {\bibinfo {volume} {51}},\ \bibinfo {pages} {053001} (\bibinfo {year}
  {2017})}\BibitemShut {NoStop}%
\bibitem [{\citenamefont {Orlandini}\ and\ \citenamefont
  {Whittington}(2007)}]{orlandiniwhittington}%
  \BibitemOpen
  \bibfield  {author} {\bibinfo {author} {\bibfnamefont {E.}~\bibnamefont
  {Orlandini}}\ and\ \bibinfo {author} {\bibfnamefont {S.~G.}\ \bibnamefont
  {Whittington}},\ }\bibfield  {title} {\enquote {\bibinfo {title} {Statistical
  topology of closed curves: Some applications in polymer physics},}\
  }\href@noop {} {\bibfield  {journal} {\bibinfo  {journal} {Reviews of modern
  physics}\ }\textbf {\bibinfo {volume} {79}},\ \bibinfo {pages} {611}
  (\bibinfo {year} {2007})}\BibitemShut {NoStop}%
\bibitem [{\citenamefont {des Cloizeaux}(1981)}]{descloizeaux}%
  \BibitemOpen
  \bibfield  {author} {\bibinfo {author} {\bibfnamefont {J.}~\bibnamefont {des
  Cloizeaux}},\ }\bibfield  {title} {\enquote {\bibinfo {title} {Ring polymers
  in solution: topological effects},}\ }\href@noop {} {\bibfield  {journal}
  {\bibinfo  {journal} {Journal de Physique Lettres}\ }\textbf {\bibinfo
  {volume} {42}},\ \bibinfo {pages} {433--436} (\bibinfo {year}
  {1981})}\BibitemShut {NoStop}%
\bibitem [{\citenamefont {Orlandini}\ \emph {et~al.}(1998)\citenamefont
  {Orlandini}, \citenamefont {Tesi}, \citenamefont {Van~Rensburg},\ and\
  \citenamefont {Whittington}}]{orlandinitesi}%
  \BibitemOpen
  \bibfield  {author} {\bibinfo {author} {\bibfnamefont {E.}~\bibnamefont
  {Orlandini}}, \bibinfo {author} {\bibfnamefont {M.~C.}\ \bibnamefont {Tesi}},
  \bibinfo {author} {\bibfnamefont {E.~J.}\ \bibnamefont {Van~Rensburg}}, \
  and\ \bibinfo {author} {\bibfnamefont {S.~G.}\ \bibnamefont {Whittington}},\
  }\bibfield  {title} {\enquote {\bibinfo {title} {Asymptotics of knotted
  lattice polygons},}\ }\href@noop {} {\bibfield  {journal} {\bibinfo
  {journal} {Journal of Physics A: Mathematical and General}\ }\textbf
  {\bibinfo {volume} {31}},\ \bibinfo {pages} {5953} (\bibinfo {year}
  {1998})}\BibitemShut {NoStop}%
\bibitem [{\citenamefont {Deutsch}(1999)}]{deutsch}%
  \BibitemOpen
  \bibfield  {author} {\bibinfo {author} {\bibfnamefont {J.}~\bibnamefont
  {Deutsch}},\ }\bibfield  {title} {\enquote {\bibinfo {title} {Equilibrium
  size of large ring molecules},}\ }\href@noop {} {\bibfield  {journal}
  {\bibinfo  {journal} {Physical Review E}\ }\textbf {\bibinfo {volume} {59}},\
  \bibinfo {pages} {R2539} (\bibinfo {year} {1999})}\BibitemShut {NoStop}%
\bibitem [{\citenamefont {Grosberg}(2000)}]{grosberg}%
  \BibitemOpen
  \bibfield  {author} {\bibinfo {author} {\bibfnamefont {A.~Y.}\ \bibnamefont
  {Grosberg}},\ }\bibfield  {title} {\enquote {\bibinfo {title} {Critical
  exponents for random knots},}\ }\href@noop {} {\bibfield  {journal} {\bibinfo
   {journal} {Physical review letters}\ }\textbf {\bibinfo {volume} {85}},\
  \bibinfo {pages} {3858} (\bibinfo {year} {2000})}\BibitemShut {NoStop}%
\bibitem [{\citenamefont {Shimamura}\ and\ \citenamefont
  {Deguchi}(2002)}]{shimamura}%
  \BibitemOpen
  \bibfield  {author} {\bibinfo {author} {\bibfnamefont {M.~K.}\ \bibnamefont
  {Shimamura}}\ and\ \bibinfo {author} {\bibfnamefont {T.}~\bibnamefont
  {Deguchi}},\ }\bibfield  {title} {\enquote {\bibinfo {title} {Anomalous
  finite-size effects for the mean-squared gyration radius of gaussian random
  knots},}\ }\href@noop {} {\bibfield  {journal} {\bibinfo  {journal} {Journal
  of Physics A: Mathematical and General}\ }\textbf {\bibinfo {volume} {35}},\
  \bibinfo {pages} {L241} (\bibinfo {year} {2002})}\BibitemShut {NoStop}%
\bibitem [{\citenamefont {Dobay}\ \emph {et~al.}(2003)\citenamefont {Dobay},
  \citenamefont {Dubochet}, \citenamefont {Millett}, \citenamefont {Sottas},\
  and\ \citenamefont {Stasiak}}]{dobay}%
  \BibitemOpen
  \bibfield  {author} {\bibinfo {author} {\bibfnamefont {A.}~\bibnamefont
  {Dobay}}, \bibinfo {author} {\bibfnamefont {J.}~\bibnamefont {Dubochet}},
  \bibinfo {author} {\bibfnamefont {K.}~\bibnamefont {Millett}}, \bibinfo
  {author} {\bibfnamefont {P.-E.}\ \bibnamefont {Sottas}}, \ and\ \bibinfo
  {author} {\bibfnamefont {A.}~\bibnamefont {Stasiak}},\ }\bibfield  {title}
  {\enquote {\bibinfo {title} {Scaling behavior of random knots},}\ }\href@noop
  {} {\bibfield  {journal} {\bibinfo  {journal} {Proceedings of the National
  Academy of Sciences}\ }\textbf {\bibinfo {volume} {100}},\ \bibinfo {pages}
  {5611--5615} (\bibinfo {year} {2003})}\BibitemShut {NoStop}%
\bibitem [{\citenamefont {Matsuda}\ \emph {et~al.}(2003)\citenamefont
  {Matsuda}, \citenamefont {Yao}, \citenamefont {Tsukahara}, \citenamefont
  {Deguchi}, \citenamefont {Furuta},\ and\ \citenamefont {Inami}}]{matsuda}%
  \BibitemOpen
  \bibfield  {author} {\bibinfo {author} {\bibfnamefont {H.}~\bibnamefont
  {Matsuda}}, \bibinfo {author} {\bibfnamefont {A.}~\bibnamefont {Yao}},
  \bibinfo {author} {\bibfnamefont {H.}~\bibnamefont {Tsukahara}}, \bibinfo
  {author} {\bibfnamefont {T.}~\bibnamefont {Deguchi}}, \bibinfo {author}
  {\bibfnamefont {K.}~\bibnamefont {Furuta}}, \ and\ \bibinfo {author}
  {\bibfnamefont {T.}~\bibnamefont {Inami}},\ }\bibfield  {title} {\enquote
  {\bibinfo {title} {Average size of random polygons with fixed knot
  topology},}\ }\href@noop {} {\bibfield  {journal} {\bibinfo  {journal}
  {Physical Review E}\ }\textbf {\bibinfo {volume} {68}},\ \bibinfo {pages}
  {011102} (\bibinfo {year} {2003})}\BibitemShut {NoStop}%
\bibitem [{\citenamefont {Moore}, \citenamefont {Lua},\ and\ \citenamefont
  {Yu.~Grosberg}(2005)}]{lgm1}%
  \BibitemOpen
  \bibfield  {author} {\bibinfo {author} {\bibfnamefont {N.~T.}\ \bibnamefont
  {Moore}}, \bibinfo {author} {\bibfnamefont {R.~C.}\ \bibnamefont {Lua}}, \
  and\ \bibinfo {author} {\bibfnamefont {A.}~\bibnamefont {Yu.~Grosberg}},\
  }\bibfield  {title} {\enquote {\bibinfo {title} {Under-knotted and
  over-knotted polymers 1: Unrestricted loops},}\ }in\ \href@noop {} {\emph
  {\bibinfo {booktitle} {Physical and Numerical Models in Knot Theory:
  Including Applications to the Life Sciences}}}\ (\bibinfo  {publisher} {World
  Scientific},\ \bibinfo {year} {2005})\ pp.\ \bibinfo {pages}
  {363--384}\BibitemShut {NoStop}%
\bibitem [{\citenamefont {Lua}, \citenamefont {Moore},\ and\ \citenamefont
  {Yu.~Grosberg}(2005)}]{lgm2}%
  \BibitemOpen
  \bibfield  {author} {\bibinfo {author} {\bibfnamefont {R.~C.}\ \bibnamefont
  {Lua}}, \bibinfo {author} {\bibfnamefont {N.~T.}\ \bibnamefont {Moore}}, \
  and\ \bibinfo {author} {\bibfnamefont {A.}~\bibnamefont {Yu.~Grosberg}},\
  }\bibfield  {title} {\enquote {\bibinfo {title} {Under-knotted and
  over-knotted polymers 2: Compact self-avoiding loops},}\ }in\ \href@noop {}
  {\emph {\bibinfo {booktitle} {Physical and Numerical Models in Knot Theory:
  Including Applications to the Life Sciences}}}\ (\bibinfo  {publisher} {World
  Scientific},\ \bibinfo {year} {2005})\ pp.\ \bibinfo {pages}
  {385--398}\BibitemShut {NoStop}%
\bibitem [{\citenamefont {Uehara}\ and\ \citenamefont
  {Deguchi}(2017)}]{uehara}%
  \BibitemOpen
  \bibfield  {author} {\bibinfo {author} {\bibfnamefont {E.}~\bibnamefont
  {Uehara}}\ and\ \bibinfo {author} {\bibfnamefont {T.}~\bibnamefont
  {Deguchi}},\ }\bibfield  {title} {\enquote {\bibinfo {title} {Scaling
  behavior of knotted random polygons and self-avoiding polygons: Topological
  swelling with enhanced exponent},}\ }\href@noop {} {\bibfield  {journal}
  {\bibinfo  {journal} {The Journal of chemical physics}\ }\textbf {\bibinfo
  {volume} {147}},\ \bibinfo {pages} {214901} (\bibinfo {year}
  {2017})}\BibitemShut {NoStop}%
\bibitem [{\citenamefont {Flory}(1953)}]{flory}%
  \BibitemOpen
  \bibfield  {author} {\bibinfo {author} {\bibfnamefont {P.~J.}\ \bibnamefont
  {Flory}},\ }\href@noop {} {\emph {\bibinfo {title} {Principles of polymer
  chemistry}}}\ (\bibinfo  {publisher} {Cornell University Press},\ \bibinfo
  {year} {1953})\BibitemShut {NoStop}%
\bibitem [{\citenamefont {De~Gennes}\ and\ \citenamefont
  {Gennes}(1979)}]{degennes}%
  \BibitemOpen
  \bibfield  {author} {\bibinfo {author} {\bibfnamefont {P.-G.}\ \bibnamefont
  {De~Gennes}}\ and\ \bibinfo {author} {\bibfnamefont {P.-G.}\ \bibnamefont
  {Gennes}},\ }\href@noop {} {\emph {\bibinfo {title} {Scaling concepts in
  polymer physics}}}\ (\bibinfo  {publisher} {Cornell university press},\
  \bibinfo {year} {1979})\BibitemShut {NoStop}%
\bibitem [{\citenamefont {Doi}\ and\ \citenamefont {Edwards}(1988)}]{doi}%
  \BibitemOpen
  \bibfield  {author} {\bibinfo {author} {\bibfnamefont {M.}~\bibnamefont
  {Doi}}\ and\ \bibinfo {author} {\bibfnamefont {S.~F.}\ \bibnamefont
  {Edwards}},\ }\href@noop {} {\emph {\bibinfo {title} {The theory of polymer
  dynamics}}},\ Vol.~\bibinfo {volume} {73}\ (\bibinfo  {publisher} {oxford
  university press},\ \bibinfo {year} {1988})\BibitemShut {NoStop}%
\bibitem [{\citenamefont {Ma}(2013)}]{ma}%
  \BibitemOpen
  \bibfield  {author} {\bibinfo {author} {\bibfnamefont {J.}~\bibnamefont
  {Ma}},\ }\bibfield  {title} {\enquote {\bibinfo {title} {Components of random
  links},}\ }\href@noop {} {\bibfield  {journal} {\bibinfo  {journal} {Journal
  of Knot Theory and Its Ramifications}\ }\textbf {\bibinfo {volume} {22}},\
  \bibinfo {pages} {1350043} (\bibinfo {year} {2013})}\BibitemShut {NoStop}%
\bibitem [{\citenamefont {Vologodskii}\ and\ \citenamefont
  {Cozzarelli}(1994)}]{vologodskiicozzarelli}%
  \BibitemOpen
  \bibfield  {author} {\bibinfo {author} {\bibfnamefont {A.~V.}\ \bibnamefont
  {Vologodskii}}\ and\ \bibinfo {author} {\bibfnamefont {N.~R.}\ \bibnamefont
  {Cozzarelli}},\ }\bibfield  {title} {\enquote {\bibinfo {title}
  {Conformational and thermodynamic properties of supercoiled {DNA}},}\
  }\href@noop {} {\bibfield  {journal} {\bibinfo  {journal} {Annual review of
  biophysics and biomolecular structure}\ }\textbf {\bibinfo {volume} {23}},\
  \bibinfo {pages} {609--643} (\bibinfo {year} {1994})}\BibitemShut {NoStop}%
\bibitem [{\citenamefont {Liu}\ and\ \citenamefont {Wang}(1987)}]{liuwang}%
  \BibitemOpen
  \bibfield  {author} {\bibinfo {author} {\bibfnamefont {L.~F.}\ \bibnamefont
  {Liu}}\ and\ \bibinfo {author} {\bibfnamefont {J.~C.}\ \bibnamefont {Wang}},\
  }\bibfield  {title} {\enquote {\bibinfo {title} {Supercoiling of the {DNA}
  template during transcription},}\ }\href@noop {} {\bibfield  {journal}
  {\bibinfo  {journal} {Proceedings of the National Academy of Sciences}\
  }\textbf {\bibinfo {volume} {84}},\ \bibinfo {pages} {7024--7027} (\bibinfo
  {year} {1987})}\BibitemShut {NoStop}%
\bibitem [{\citenamefont {Champoux}\ and\ \citenamefont
  {Been}(1980)}]{champouxbeen}%
  \BibitemOpen
  \bibfield  {author} {\bibinfo {author} {\bibfnamefont {J.~J.}\ \bibnamefont
  {Champoux}}\ and\ \bibinfo {author} {\bibfnamefont {M.~D.}\ \bibnamefont
  {Been}},\ }\bibfield  {title} {\enquote {\bibinfo {title} {Topoisomerases and
  the swivel problem},}\ }in\ \href@noop {} {\emph {\bibinfo {booktitle}
  {Mechanistic studies of {DNA} replication and genetic recombination}}}\
  (\bibinfo  {publisher} {Elsevier},\ \bibinfo {year} {1980})\ pp.\ \bibinfo
  {pages} {809--815}\BibitemShut {NoStop}%
\bibitem [{\citenamefont {Deweese}, \citenamefont {Osheroff},\ and\
  \citenamefont {Osheroff}(2009)}]{deweeseosheroff}%
  \BibitemOpen
  \bibfield  {author} {\bibinfo {author} {\bibfnamefont {J.~E.}\ \bibnamefont
  {Deweese}}, \bibinfo {author} {\bibfnamefont {M.~A.}\ \bibnamefont
  {Osheroff}}, \ and\ \bibinfo {author} {\bibfnamefont {N.}~\bibnamefont
  {Osheroff}},\ }\bibfield  {title} {\enquote {\bibinfo {title} {{DNA} topology
  and topoisomerases: teaching a ``knotty'' subject},}\ }\href@noop {}
  {\bibfield  {journal} {\bibinfo  {journal} {Biochemistry and Molecular
  Biology Education}\ }\textbf {\bibinfo {volume} {37}},\ \bibinfo {pages}
  {2--10} (\bibinfo {year} {2009})}\BibitemShut {NoStop}%
\bibitem [{\citenamefont {Espeli}\ and\ \citenamefont
  {Marians}(2004)}]{espelimarians}%
  \BibitemOpen
  \bibfield  {author} {\bibinfo {author} {\bibfnamefont {O.}~\bibnamefont
  {Espeli}}\ and\ \bibinfo {author} {\bibfnamefont {K.~J.}\ \bibnamefont
  {Marians}},\ }\bibfield  {title} {\enquote {\bibinfo {title} {Untangling
  intracellular {DNA} topology},}\ }\href@noop {} {\bibfield  {journal}
  {\bibinfo  {journal} {Molecular microbiology}\ }\textbf {\bibinfo {volume}
  {52}},\ \bibinfo {pages} {925--931} (\bibinfo {year} {2004})}\BibitemShut
  {NoStop}%
\bibitem [{\citenamefont {Seol}\ and\ \citenamefont
  {Neuman}(2016)}]{seolneuman}%
  \BibitemOpen
  \bibfield  {author} {\bibinfo {author} {\bibfnamefont {Y.}~\bibnamefont
  {Seol}}\ and\ \bibinfo {author} {\bibfnamefont {K.~C.}\ \bibnamefont
  {Neuman}},\ }\bibfield  {title} {\enquote {\bibinfo {title} {The dynamic
  interplay between {DNA} topoisomerases and {DNA} topology},}\ }\href@noop {}
  {\bibfield  {journal} {\bibinfo  {journal} {Biophysical reviews}\ }\textbf
  {\bibinfo {volume} {8}},\ \bibinfo {pages} {101--111} (\bibinfo {year}
  {2016})}\BibitemShut {NoStop}%
\bibitem [{\citenamefont {Darcy}\ and\ \citenamefont
  {Vazquez}(2013)}]{darcyvazquez}%
  \BibitemOpen
  \bibfield  {author} {\bibinfo {author} {\bibfnamefont {I.~K.}\ \bibnamefont
  {Darcy}}\ and\ \bibinfo {author} {\bibfnamefont {M.}~\bibnamefont
  {Vazquez}},\ }\href@noop {} {\enquote {\bibinfo {title} {Determining the
  topology of stable protein--{DNA} complexes},}\ } (\bibinfo {year}
  {2013})\BibitemShut {NoStop}%
\bibitem [{\citenamefont {Barbensi}\ \emph {et~al.}(2020)\citenamefont
  {Barbensi}, \citenamefont {Celoria}, \citenamefont {Harrington},
  \citenamefont {Stasiak},\ and\ \citenamefont {Buck}}]{barbensi2020grid}%
  \BibitemOpen
  \bibfield  {author} {\bibinfo {author} {\bibfnamefont {A.}~\bibnamefont
  {Barbensi}}, \bibinfo {author} {\bibfnamefont {D.}~\bibnamefont {Celoria}},
  \bibinfo {author} {\bibfnamefont {H.~A.}\ \bibnamefont {Harrington}},
  \bibinfo {author} {\bibfnamefont {A.}~\bibnamefont {Stasiak}}, \ and\
  \bibinfo {author} {\bibfnamefont {D.}~\bibnamefont {Buck}},\ }\bibfield
  {title} {\enquote {\bibinfo {title} {Grid diagrams as tools to investigate
  knot spaces and topoisomerase-mediated simplification of {DNA} topology},}\
  }\href@noop {} {\bibfield  {journal} {\bibinfo  {journal} {Science Advances}\
  }\textbf {\bibinfo {volume} {6}},\ \bibinfo {pages} {eaay1458} (\bibinfo
  {year} {2020})}\BibitemShut {NoStop}%
\bibitem [{\citenamefont {Cromwell}(1998)}]{cromwell}%
  \BibitemOpen
  \bibfield  {author} {\bibinfo {author} {\bibfnamefont {P.}~\bibnamefont
  {Cromwell}},\ }\bibfield  {title} {\enquote {\bibinfo {title} {Arc
  presentations of knots and links},}\ }\href@noop {} {\bibfield  {journal}
  {\bibinfo  {journal} {Banach Center Publications}\ }\textbf {\bibinfo
  {volume} {42}},\ \bibinfo {pages} {57--64} (\bibinfo {year}
  {1998})}\BibitemShut {NoStop}%
\bibitem [{\citenamefont {Doig}(2019)}]{dmt}%
  \BibitemOpen
  \bibfield  {author} {\bibinfo {author} {\bibfnamefont {M.}~\bibnamefont
  {Doig}},\ }\href@noop {} {\enquote {\bibinfo {title} {Doig {M}ath
  {T}oolkit},}\ }\bibinfo {howpublished} {\url{http://doigmath.maderak.com}}
  (\bibinfo {year} {2019})\BibitemShut {NoStop}%
\bibitem [{\citenamefont {Witte}(2019)}]{witte2019link}%
  \BibitemOpen
  \bibfield  {author} {\bibinfo {author} {\bibfnamefont {S.~L.}\ \bibnamefont
  {Witte}},\ }\emph {\bibinfo {title} {Link Nomenclature, Random Grid Diagrams,
  and Markov Chain Methods in Knot Theory}},\ \href@noop {} {Ph.D. thesis},\
  \bibinfo  {school} {UNIVERSITY OF CALIFORNIA DAVIS} (\bibinfo {year}
  {2019})\BibitemShut {NoStop}%
\bibitem [{oei()}]{oeis}%
  \BibitemOpen
  \href@noop {} {\enquote {\bibinfo {title} {The {O}n-{L}ine {E}ncyclopedia of
  {I}nteger {S}equences},}\ }\bibinfo {howpublished}
  {\url{http://oeis.org}}\BibitemShut {NoStop}%
\bibitem [{\citenamefont {Delbr{\"u}ck}(1962)}]{delbruck1962knotting}%
  \BibitemOpen
  \bibfield  {author} {\bibinfo {author} {\bibfnamefont {M.}~\bibnamefont
  {Delbr{\"u}ck}},\ }\bibfield  {title} {\enquote {\bibinfo {title} {Knotting
  problems in biology},}\ }in\ \href@noop {} {\emph {\bibinfo {booktitle}
  {Proc. Symp. Appl. Math}}},\ Vol.~\bibinfo {volume} {14}\ (\bibinfo {year}
  {1962})\ pp.\ \bibinfo {pages} {55--63}\BibitemShut {NoStop}%
\bibitem [{\citenamefont {Frisch}\ and\ \citenamefont
  {Wasserman}(1961)}]{frisch1961chemical}%
  \BibitemOpen
  \bibfield  {author} {\bibinfo {author} {\bibfnamefont {H.~L.}\ \bibnamefont
  {Frisch}}\ and\ \bibinfo {author} {\bibfnamefont {E.}~\bibnamefont
  {Wasserman}},\ }\bibfield  {title} {\enquote {\bibinfo {title} {Chemical
  topology 1},}\ }\href@noop {} {\bibfield  {journal} {\bibinfo  {journal}
  {Journal of the American Chemical Society}\ }\textbf {\bibinfo {volume}
  {83}},\ \bibinfo {pages} {3789--3795} (\bibinfo {year} {1961})}\BibitemShut
  {NoStop}%
\bibitem [{\citenamefont {Sumners}\ and\ \citenamefont
  {Whittington}(1988)}]{Sumners_1988}%
  \BibitemOpen
  \bibfield  {author} {\bibinfo {author} {\bibfnamefont {D.~W.}\ \bibnamefont
  {Sumners}}\ and\ \bibinfo {author} {\bibfnamefont {S.~G.}\ \bibnamefont
  {Whittington}},\ }\bibfield  {title} {\enquote {\bibinfo {title} {Knots in
  self-avoiding walks},}\ }\href {\doibase 10.1088/0305-4470/21/7/030}
  {\bibfield  {journal} {\bibinfo  {journal} {Journal of Physics A:
  Mathematical and General}\ }\textbf {\bibinfo {volume} {21}},\ \bibinfo
  {pages} {1689--1694} (\bibinfo {year} {1988})}\BibitemShut {NoStop}%
\bibitem [{\citenamefont {Pippenger}(1989)}]{pippenger1989knots}%
  \BibitemOpen
  \bibfield  {author} {\bibinfo {author} {\bibfnamefont {N.}~\bibnamefont
  {Pippenger}},\ }\bibfield  {title} {\enquote {\bibinfo {title} {Knots in
  random walks},}\ }\href@noop {} {\bibfield  {journal} {\bibinfo  {journal}
  {Discrete Applied Mathematics}\ }\textbf {\bibinfo {volume} {25}},\ \bibinfo
  {pages} {273--278} (\bibinfo {year} {1989})}\BibitemShut {NoStop}%
\bibitem [{\citenamefont {Even-Zohar}(2017{\natexlab{b}})}]{even2017writhe}%
  \BibitemOpen
  \bibfield  {author} {\bibinfo {author} {\bibfnamefont {C.}~\bibnamefont
  {Even-Zohar}},\ }\bibfield  {title} {\enquote {\bibinfo {title} {The writhe
  of permutations and random framed knots},}\ }\href@noop {} {\bibfield
  {journal} {\bibinfo  {journal} {Random Structures \& Algorithms}\ }\textbf
  {\bibinfo {volume} {51}},\ \bibinfo {pages} {121--142} (\bibinfo {year}
  {2017}{\natexlab{b}})}\BibitemShut {NoStop}%
\bibitem [{\citenamefont {Tucker}(1994)}]{tucker}%
  \BibitemOpen
  \bibfield  {author} {\bibinfo {author} {\bibfnamefont {A.}~\bibnamefont
  {Tucker}},\ }\href@noop {} {\emph {\bibinfo {title} {Applied
  combinatorics}}}\ (\bibinfo  {publisher} {John Wiley \& Sons, Inc.},\
  \bibinfo {year} {1994})\BibitemShut {NoStop}%
\bibitem [{\citenamefont {Wilf}(2005)}]{wilf}%
  \BibitemOpen
  \bibfield  {author} {\bibinfo {author} {\bibfnamefont {H.~S.}\ \bibnamefont
  {Wilf}},\ }\href@noop {} {\emph {\bibinfo {title}
  {generatingfunctionology}}}\ (\bibinfo  {publisher} {AK Peters/CRC Press},\
  \bibinfo {year} {2005})\BibitemShut {NoStop}%
\end{thebibliography}%
